\documentclass[hidelinks]{siamart250106}
\usepackage{fullpage}

\usepackage{hyperref}
\hypersetup{
    colorlinks,
    citecolor=green,
    filecolor=black,
    linkcolor=blue,
    urlcolor=blue
}
\hypersetup{linktocpage}
\usepackage[final]{showlabels}

\usepackage{mathtools}
\usepackage{cite}
\usepackage{graphicx}
 
\usepackage{cleveref}
\usepackage{amsbsy}
\usepackage{latexsym}
\usepackage{amsfonts}
\usepackage{amssymb}
\usepackage{upgreek}
\usepackage{amsmath}
\usepackage{enumerate}
\usepackage{subcaption}
\usepackage{stmaryrd}

\usepackage{amssymb,amsmath,graphicx,amscd,thmtools,tikz}
\usepackage{pgfplots}
\pgfplotsset{compat = 1.18}

\usetikzlibrary{shapes,snakes}
\usetikzlibrary{cd}
\usepackage{csvsimple}
\usepackage{siunitx,array,booktabs}
\usepackage{dsfont}

\newcommand{\He}{\mathsf H}
\newcommand{\T}{\mathsf T}

\providecommand{\norm}[1]{\lVert#1\rVert}

\DeclareMathOperator{\diag}{Diag}

\DeclareMathOperator{\rank}{rank}

\DeclareMathOperator{\sign}{sign}

\DeclareMathOperator{\conv}{conv}

\newcommand{\abs}[1]{{\left\lvert #1 \right\rvert}}

\DeclareMathOperator{\R}{\mathbb{R}}
\DeclareMathOperator{\C}{\mathbb{C}}
\DeclareMathOperator{\N}{\mathbb{N}}

\let\S\relax\DeclareMathOperator{\S}{\mathbb{S}}

\newcommand{\lem}[1]{{Lemma\,#1}}

\newcommand{\pr}[1]{{Proposition\,#1}}
\newcommand{\theo}[1]{{Theorem\,#1}}
\newcommand{\sect}[1]{{Section\,#1}}
\newcommand{\ch}[1]{{Chapter\,#1}}
\newcommand{\coro}[1]{{Corollary\,#1}}
\newcommand{\examp}[1]{{Example\,#1}}

\numberwithin{equation}{section}

\theoremstyle{plain}

\usepackage{algorithmicx}
\usepackage{algpseudocode}

\usepackage{todonotes}

\title{A Newton method for solving locally definite multiparameter eigenvalue problems by multiindex}
\author{Henrik Eisenmann \thanks{
    Institut f\"ur Geometrie und Praktische Mathematik, RWTH Aachen University, Templergraben 55, 52062 Aachen, Germany
(\email{eisenmann@igpm.rwth-aachen.de})
\funding{The work of H.E.~was funded by Deutsche Forschungs\-gemeinschaft – project number 501389786.}}
}

\date{\today}

\begin{document}

\maketitle

\begin{abstract}
    We present a new approach to compute eigenvalues and eigenvectors of locally definite multiparameter eigenvalue problems by its signed multiindex. The method has the interpretation of a semismooth Newton method applied to certain functions that have a unique zero. We can therefore show local quadratic convergence, and for certain extreme eigenvalues even global linear convergence of the method. Local definiteness is a weaker condition than right and left definiteness, which is often considered for multiparameter eigenvalue problems. These conditions are naturally satisfied for multiparameter Sturm-Liouville problems that arise when separation of variables can be applied to multidimensional boundary eigenvalue problems.
\end{abstract}



\section{Introduction}

In this work, we are interested 
in the \emph{multiparameter eigenvalue problem}~(MEP), which is an algebraic system of equations consisting of eigenvalue problems. Let  $A_{k\ell}\in \C^{n_k\times n_k}$ be square matrices. A solution of the multiparameter eigenvalue problem is given by $u_k\in\mathcal U_{k}\subset \C^{n_k}$, where $
\mathcal U_k$ consists of $n_k$ dimensional unit vectors, and $(\lambda_1,\ldots, \lambda_m)\in\C^m$ that satisfy
\begin{equation}\label{eq:MEP}
\setlength\arraycolsep{1pt}
\begin{array}{cccccccccc}
    \bigl(A_{10}& +&\lambda_1 A_{11}
    &+&\ldots&+& \lambda_m A_{1m}\bigr)&u_1&=&0\\
    \bigl(A_{20}& +&\lambda_1 A_{21}
    &+&\ldots&+& \lambda_m A_{2m}\bigr)&u_2&=&0\\
    \vdots &&\vdots&&&&\vdots&\vdots&&\vdots\\
    \bigl(A_{m0}& +&\lambda_1 A_{m1}
    &+&\ldots&+& \lambda_m A_{mm}\bigr)&u_m&=&0
\end{array} 
\end{equation}
or for $(\lambda_0,\ldots, \lambda_m)\neq 0$ in the homogeneous version
\begin{equation}\label{eq:HomMEP}
\setlength\arraycolsep{1pt}
\begin{array}{cccccccccc}
    \bigl(\lambda_0A_{10}& +&\lambda_1 A_{11}
    &+&\ldots&+& \lambda_m A_{1m}\bigr)&u_1&=&0\\
    \bigl(\lambda_0A_{20}& +&\lambda_1 A_{21}
    &+&\ldots&+& \lambda_m A_{2m}\bigr)&u_2&=&0\\
    \vdots &&\vdots&&&&\vdots&\vdots&&\vdots\\
    \bigl(\lambda_0A_{m0}& +&\lambda_1 A_{m1}
    &+&\ldots&+& \lambda_m A_{mm}\bigr)&u_m&=&0,
\end{array} 
\end{equation} 
We call $u_1\otimes \ldots\otimes u_m$ an eigenvector corresponding to the eigenvalue  $\lambda=(\lambda_1,\ldots,\lambda_m)$ in case of~\eqref{eq:MEP} and the eigenvalue  $\lambda=(\lambda_0,\ldots,\lambda_m)$ in case of~\eqref{eq:HomMEP}. The multiparameter eigenvalue problem combines linear systems of equations and eigenvalue problems. In the case $m=1$, equation~\eqref{eq:MEP} is a generalized eigenvalue problem, and if $n_k=1$ for $k=1,\ldots,m$, then~\eqref{eq:MEP} reduces to a linear system of equations.

The solutions to the MEP can be obtained using multilinear algebra techniques; see e.g.,~\cite[\theo{6.8.1}]{Atkinson1972}. The essence of the idea is to apply a multilinear version of Cramer's rule for linear systems. 
To apply Cramer's rule, we define the matrix 
\[
\tilde W(v_1,\ldots,v_m,u_1,\ldots,u_m)=[\tilde w_{k\ell}(u_k,v_k)]_{k=1,\ldots, m}^{\ell=0,\ldots,m}=\left[v_k^\He A^{}_{k\ell}u_k^{}\right]_{k=1,\ldots, m}^{\ell=0,\ldots,m}
\]
for $u_k,v_k\in\C^{n_k}$. Now assuming that $u_1\otimes \ldots\otimes u_m$ is an eigenvector corresponding to $\lambda=(\lambda_0,\ldots,\lambda_m)$, then the eigenvalue~$\lambda$ satisfies the linear equation
\[
\tilde W(v_1,\ldots,v_m,u_1,\ldots,u_m) \lambda =0
\]
for any $v_k\in\C^{n_k}$.
Now assume there is $\mu=(\mu_0,\ldots,\mu_m)$ such that for all  $u_k\in\mathcal U_{k}$ the matrix
\[
\begin{pmatrix}
\mu^\T\\
\tilde W(v_1,\ldots,v_m,u_1,\ldots,u_m)
\end{pmatrix}
\]
is invertible for some $v_k\in \C^{n_k}$.
We define 
linear operators $\Delta,\Delta_0,\dots,\Delta_m\colon\bigotimes_{k=1}^m \C^{n_k}\to\bigotimes_{k=1}^m \C^{n_k}$ by
\begin{equation}\label{eq:deltaMatrix}
    (v_1\otimes \ldots \otimes v_m)^\He \Delta (u_1\otimes  \ldots\otimes u_m)
    =
    \det
    \begin{pmatrix}
    \mu^\T\\
    \tilde W(v_1,\ldots,v_m,u_1,\ldots,u_m)
    \end{pmatrix}
\end{equation}
using linearity for each $u_k, v_k\in \C^{n_k}$, and in a similar way
\[
    (v_1\otimes \ldots \otimes v_m)^\He \Delta_\ell (u_1\otimes \ldots \otimes u_m)
    =
    \det
    \begin{pmatrix}
    e_\ell\\
    \tilde W(v_1,\ldots,v_m,u_1,\ldots,u_m)
    \end{pmatrix},
\]
where $e_\ell$ is the coordinate vector with $1$ at its $\ell+1$-th entry.
Then by Cramer's rule, the eigenpair $(\lambda,u_1\otimes \ldots \otimes u_m)$ satisfies
\begin{equation}\label{eq:MEPDelta}
    \mu^\T \lambda \;\Delta_\ell(u_1\otimes \ldots \otimes u_m)=\lambda_\ell\, \Delta (u_1\otimes \ldots \otimes u_m).
\end{equation}
Hence, all eigenpairs of the MEP can be obtained as the solutions to the simultaneous eigenvalue problems~\eqref{eq:MEPDelta}. In~\cite[\ch 6]{Atkinson1972} the converse is also shown.
In addition, the $m+1$~linear operators $\Delta^{-1}\Delta_\ell$ commute, and thus they share the same eigenvectors $u_1\otimes \ldots \otimes u_m$, which are rank-one tensors.

Multiparameter eigenvalue problems have undergone extensive investigation;
see e.g.,~\cite{Atkinson1972,Volkmer88,Faierman1991,AtkinsonMingarelli2011,Sleeman1978} for books considering this problem class.
These problems naturally emerge in mathematical physics where variables allow separation, but the resulting spectral parameters do not.
Various applications, such as delay differential equations~\cite{JH2009}
and optimization problems~\cite{SNTI2016}, lead to formulations resembling~\eqref{eq:MEP} or~\eqref{eq:HomMEP}.
Another context involves the decomposition of the domain in a boundary eigenvalue problem~\cite[\sect 5.2]{RJ2021}. 
Additionally, it is noteworthy that certain nonlinear eigenvalue problems can be represented using~\eqref{eq:MEP}, as elucidated in~\cite{RJ2021}.

There are various approaches to solve~\eqref{eq:MEP}, with many leveraging the linear eigenvalue problem~\eqref{eq:MEPDelta}. 
One option is using the generalized Schur decomposition. This was initially considered in~\cite{ST1986},
for the case that all matrices are Hermitian and $\Delta_0$ is positive definite and for the more general case, that $\Delta_0$ is invertible in~\cite{HKP2004}. 
These strategies perform well when $\prod_{k=1}^m n_k$ is not excessively large, as they have a complexity of order~$\mathcal O(\prod_{k=1}^m n_k^3)$. This is only feasible for small $n_k$ and $m$, as otherwise $\prod_{k=1}^m n_k$ becomes too large to feasibly solve eigenvalue problems. 
For $m=2$ and larger $n_k$, subspace methods come into play to find a selection of eigenvalues. In~\cite{HP2002} and~\cite{HKP2004} a Jacobi-Davidson type method for the two-parameter case was proposed and in~\cite{MP2015} an Arnoldi type method was considered. 
In the case of $m = 3$, various subspace methods were introduced in~\cite{HMMP2019}. 
An alternative approach involves homotopy continuation, as discussed in~\cite{Plestenjak2000, Plestenjak2001,DYY2016} and~\cite{RodriguezDuYouLimFiber2021}.
Continuation methods aim to find all eigenvalues.

In this work, we present a novel approach for addressing multiparameter eigenvalue problems through a Newton-type method based on the eigenvalue's multiindex, as explicitly defined in \Cref{sec:definiteMEPs}.
The multiindex of an eigenvalue is a concept usually used in the context of multiparameter Sturm-Liouville problems. It is however also a very useful notion for locally definite multiparameter eigenvalue. The convergence analysis of our resulting method is carried out for locally definite multiparameter eigenvalue with finite dimensional Hermitian matrices.
The essential requirement of local definiteness is elaborated in the subsequent section.
The local definiteness requirement is stronger than the assumption that the linear operator~$\Delta$ in~\eqref{eq:deltaMatrix} is invertible for some $\mu$; see e.g.,~\cite[\theo{10.4.1}]{Atkinson1972}.
Local definiteness criteria are met within a specific class of boundary eigenvalue problems, formally introduced in \Cref{sec:Sturm-Liouville}.
We develop a special Newton method tailored to locally definite multiparameter eigenvalues problems  in \Cref{sec:NewtonMEP} and discuss a perspective coming from optimization problems in \Cref{sec:OptMEP}.
Notably, these methods extend their applicability to multiparameter eigenvalue problems with large values of $n_k$ and $m$ when targeting only a subset of eigenvalues associated with certain multiindices. The computational complexity per eigenvalue remains at $\mathcal O(\sum_{k=1}^m{n_k^3}+ m^3)$.
This expression is dominated by $\mathcal O(\sum_{k=1}^m{n_k^3})$ if $m$ is small compared to the $n_k$. If additionally $n_k\eqsim n$, we get costs of the order $\mathcal O(m{n^3})$ which is linear in $m$.
The performance of this method is shown in numerical experiments, as elaborated in \Cref{sec:numericalMEP}. It is important to note that the foundational concepts discussed herein were primarily expounded in the thesis~\cite[\ch 3]{eisenmann}. This work extends the results of the thesis by highlighting how to apply the method constructed in \Cref{sec:NewtonMEP} to a discretization of multiparameter Sturm-Liouville problems, even if the discretization involves non-Hermitian matrices as explained in \Cref{sec: non Hermitian}. We also made  more thorough numerical experiments, exploring the influence of different parameters on the performance of the method in \Cref{sec:numericalMEP}.
The resulting method exhibits a close relation to the alternating method proposed in~\cite{eisenmann2021solving}, and this connection is further discussed in~\cite[\ch 3]{eisenmann}.

The utilization of Newton's method for solving multiparameter eigenvalue problems (MEPs) is not a novel concept and has been previously proposed, as exemplified in~\cite{Bohte1982}. The approach in~\cite{Bohte1982} involves seeking the joint zeros of functions $f_k(\lambda)=\det(\sum_{\ell=0}^m \lambda_\ell A_{k\ell})$ for $k=1,\ldots,m$. However, this method, particularly for general MEPs, proves to be sensitive to initialization, necessitating a precise starting guess.
To overcome this sensitivity issue, we adopt a different strategy by applying Newton's method to functions that have exactly one zero and therefore making it less sensitive to initialization in the sense that a different initialization do not lead to different solutions if the method does converge. Our proposed methods share common ideas with approaches designed for solving multiparameter Sturm-Liouville eigenvalue problems, wherein the objective is to identify eigenfunctions possessing a specific number of internal zeros. This concept is discussed in works such as~\cite{Levitina1994,Levititna1999}. 
These methods make use of the Prüfer transformation; see e.g.,~\cite[\ch 5.3]{AtkinsonMingarelli2011} and~\cite{BindingBrwon1984}. Here, the linear second order boundary eigenvalue problem is transformed into a nonlinear first order ordinary differential equation for the oscillation. The final value then describes the kind of boundary condition and the number of oscillations of the solution. This is combined with bisecting a rectangle containing the sought eigenvalue in~\cite{Levitina1994,Levititna1999}.
A bisection approach for two-parameter Sturm-Liouville problems was also used in~\cite{Ji_1992_twoparabisection} for a finite difference discretization. Bisection methods however only converge at a linear rate.
The notion of internal zeros in eigenfunctions of Sturm-Liouville eigenvalue problems represents a specialized case of the multiindex of an eigenvalue, a concept that will be explained in more detail in \Cref{sec:Sturm-Liouville}.

\section{Definite multiparameter eigenvalue problems}\label{sec:definiteMEPs}

    In this section, our objective is to explore a generalization of the well-established principle from linear algebra, asserting that any Hermitian matrix $A=A^\He$ exclusively possesses real eigenvalues. To achieve this, we draw upon relevant findings from~\cite[\ch 1]{Volkmer88}. Throughout our discussion, we maintain the assumption that all matrices $A_{k\ell}$ in~\eqref{eq:MEP} and~\eqref{eq:HomMEP} are Hermitian. Now assume that $\lambda=(\lambda_0,\ldots,\lambda_m)$ is real. Consequently,
    \[
        \sum_{\ell=0}^m \lambda_\ell A_{k\ell}  \]
   are Hermitian matrices for $k={1,\ldots,m}$
   and all their respective eigenvalues are real. This observation forms the basis for the following definition.

    \begin{definition}{\upshape\cite[\ch 1.2]{Volkmer88}}\label{def:index}
        Let $\lambda\in \R^{m+1}$ be an eigenvalue of~\eqref{eq:HomMEP}. The multiindex of $\lambda$ is the $m$-tuple $\mathbf i=(i_1,\ldots,i_m)$ such that $0$ is the $i_k$-th largest eigenvalue of $\sum_{\ell=0}^m \lambda_\ell A_{k\ell} $.
    \end{definition}

    In the case $m=1$, a sufficient but not necessary criterion that all eigenvalues have a real representation, is that there is a linear combination $\mu_0A_{10}+\mu_1A_{11}$ that is positive definite. For MEPs, we examine the matrix 
    \begin{equation}
        W(u_1,\ldots,u_m)
        =
        \begin{pmatrix}
            u_1^\He A^{}_{10}u_1^{} & \ldots  & u_1^\He A^{}_{1m}u_1^{} \\
            \vdots &  & \vdots                         \\
            u_m^\He A^{}_{m0}u_m^{} & \ldots  & u_m^\He A^{}_{mm}   u_m^{}
        \end{pmatrix}.
    \end{equation}

    \begin{definition}{\upshape\cite[\ch 1.1]{Volkmer88}}
        The MEP~\eqref{eq:HomMEP} is called locally definite if  $\rank W(u_1,\ldots,u_m)=m$ for all $u_k\in\mathcal U_k$, $k=1,\ldots,m$.
    \end{definition}
    Consider $u_1\otimes\ldots\otimes u_m$ as an eigenvector corresponding to the eigenvalue $\lambda$. In this case, $\lambda$ is in the nullspace of $W(u_1,\ldots,u_m)$. When the multiparameter eigenvalue problem (MEP) is locally definite, this nullspace is one-dimensional. Furthermore, since $W(u_1,\ldots,u_m)$ is a real matrix, the eigenvalue can be selected as real as well.
    Through scaling, we have the flexibility to restrict the search for eigenvalues to the unit sphere $\S^m=\{\lambda\in \R^{m+1}\colon \norm{\lambda}_2=1\}$.

    By the above consideration, we may further restrict to the set
    \begin{equation}\label{eq:MEPfeasibleset}
        \mathcal{P}=\{\lambda\in\S^m\colon W(u)\lambda=0 \text{ for some $u\in \mathcal U_1\times\ldots\times \mathcal U_m$}\}.
    \end{equation}
    This set obviously has the symmetry $\mathcal P=-\mathcal P$. In fact, it is the disjoint union of two connected sets $\mathcal P^+$ and $\mathcal P^-=-\mathcal P^+$; see e.g.,~\cite[\lem 1.2.1]{Volkmer88}. Here, $\mathcal P^+$ is the set
    \[
        \mathcal{P}^+=\left\{\lambda\in\S^m\colon W(u)\lambda=0 \text{ and }
        \det
        \begin{pmatrix}
            \lambda^\T\\
            W(u)
        \end{pmatrix}>0
        \text{ for some $u\in \mathcal U_1\times\ldots\times \mathcal U_m$}  \right\}.
    \]

    Utilizing the outlined methodology, we can comprehensively characterize all eigenvalues of a locally definite MEP~\eqref{eq:HomMEP}.

    \begin{theorem}{\upshape\cite[\theo{1.4.1}]{Volkmer88}}\label{thm:signedindex}
        Let the MEP~\eqref{eq:HomMEP} be locally definite. Then for every multiindex $\mathbf i\in \{1,\ldots,n_1\}\times \ldots\times \{1,\ldots,n_m\}$ and every sign $\sigma\in\{+,-\}$ there is a unique eigenvalue $\lambda\in \mathcal P^\sigma$ of multiindex $\mathbf i$. We say $\lambda$ is of signed index~$(\mathbf i,\sigma)$. 
    \end{theorem}

    Additionally, it proves beneficial to explore stronger definiteness assumptions, which, in practical applications, are frequently met.
    \begin{definition}
        The MEP~\eqref{eq:HomMEP} is called definite with respect to $\mu\in\S^m$ if
        \[
            \det 
            \begin{pmatrix}
                \mu^\T\\
                W(u_1,\ldots,u_m)
            \end{pmatrix}>0
        \]
        for all $u_k\in\mathcal U_k$, $k=1,\ldots,m$.
    \end{definition}

    This definition is equivalent to the condition that $\mathcal P$ in~\eqref{eq:MEPfeasibleset} is the disjoint union of the connected sets
    \[
        \mathcal{P}^+=\{\lambda\in\mathcal P\colon \mu^\T\lambda>0 \}
    \]
    and $\mathcal P^-=-\mathcal P^+$; see the discussion in~\cite[\ch{1.5}]{Volkmer88}. It also follows that the linear operator $\Delta$ defined in~\eqref{eq:deltaMatrix} is positive definite; see e.g.,~\cite[\theo 4.4.1]{Volkmer88}. 
    For the 
    more standard
    inhomogeneous MEP~\eqref{eq:MEP}, definiteness with respect to $\mu=(1,0,\ldots,0)$ is naturally a favorable condition. In this case, \Cref{thm:signedindex} implies, that every eigenvalue of signed index $(\mathbf i,+)$ can be scaled to $\lambda=(1,\lambda_1,\ldots,\lambda_m)$. If the MEP is definite with respect to $(1,0,\ldots,0)$, it is called \emph{right definite}.  When dealing with a right definite MEP of the form~\eqref{eq:MEP}, we denote
    \[
        \mathcal Q = \left\{\lambda\in\R^m\colon W(u)
        \begin{pmatrix}
            1\\
            \lambda
        \end{pmatrix}
        =0 \text{ for some $u\in \mathcal U_1\times\ldots\times \mathcal U_m$}  \right\}
    \]
    as an analog of $\mathcal P$. The sets $\mathcal P^{\sigma}$ and $\mathcal Q$ serve as analogies to the numerical range of a linear operator.

    Another frequently employed assumption is referred to as \emph{left definiteness}. The MEP is left definite with respect to $\mu=(0,\mu_1,\ldots, \mu_m)$ if the matrices $A_{k0}$ are negative definite and 
    \begin{equation}\label{eq:leftdeficond}
        \det
        \begin{pmatrix}
            u_1^\He A^{}_{11} u_1^{}&
            \ldots &
            u_m^\He A^{}_{1m} u^{}_m \\
            \vdots &  &\vdots\\
            u_{k-1}^\He A^{}_{{k-1},1} u^{}_{k-1}&
            \ldots &
            u_{k-1}^\He A^{}_{{k-1},m} u^{}_{k-1} \\
            \mu_1 & \ldots  & \mu_m\\
            u_{k+1}^\He A^{}_{{k+1},1} u^{}_{k+1}&
            \ldots &
            u_{k+1}^\He A^{}_{{k+1},m}u^{}_{k+1}\\
            \vdots & &\vdots\\
            u_{m}^\He A^{}_{{m}1} u^{}_m&
            \ldots &
            u_{m}^\He A^{}_{mm}u^{}_m
        \end{pmatrix}
        >
        0
    \end{equation}
    for $k=1,\ldots,m$. Left definiteness implies definiteness with respect to $\mu$.

   A number of problems arising from applications exhibit either right or left definiteness or a combination of both, for example when the MEP arises by separation of variables as described in \Cref{sec:Sturm-Liouville}. Some interesting cases can be found in~\cite[\ch 3.7]{Volkmer88}  and\cite{Plestenjak2015,ALSW_whispering_2014}.
   To conclude this section, we aim to present a useful equivalent condition for local definiteness.

    \begin{lemma}{\upshape\cite[\theo{1.4.3}]{Volkmer88}}\label{lm:localdefinite}
        The MEP~\eqref{eq:HomMEP} is locally definite if and only if for every $(\sigma_1,\ldots,\sigma_m ) \in \{-1,1\}^m $ there is an
        $(\alpha_0,\ldots,\alpha_m) \in \R^{m+1}$ such that
        $\sum_{\ell=0}^m \sigma^{}_k u_k^\He A^{}_{k\ell}u^{}_k \alpha^{}_\ell> 0$ for all $u_k \in \mathcal U_k, k = 1,\ldots,m$.
    \end{lemma}

    \section{Multiparameter
    Sturm–Liouville
    problems}\label{sec:Sturm-Liouville}

    To provide further motivation, we consider the Helmholtz equation with Dirichlet boundary conditions:
    \begin{equation}\label{eq:helmholtz}
        \begin{aligned}
            \Delta u(x)  +\lambda\, u(x) &=0\quad &&\text{for $x\in\Omega$}\\
            u(x) &=0 \quad &&\text{for $x\in\partial \Omega$}
        \end{aligned}
    \end{equation}
    on a domain $\Omega$. 
    For our purpose, the situation is intriguing when the domain is the ellipse $\Omega=\{(x,y)\in \R^2\colon x^2/a^2+y^2/b^2<1 \}$ with $b>a>0$. In the elliptical coordinates
    \[
    x=c\cosh(\rho)\cos(\varphi),\quad y=c\sinh(\rho)\sin(\varphi),
    \]
    the equation separates into
    \begin{equation*}\label{eq:mathieu}
    \begin{array}{ccccccccl}
         P''(\rho)&+&\lambda\, \frac{c^2}{2} \cosh (2\rho)\, P(\rho)& -& \nu\, P(\rho)&=&0 &&\text{for $\rho\in(0,r)$}, \\
         \Phi''(\varphi)&-&\lambda\, \frac{c^2}{2} \cos (2\varphi)\, \Phi(\varphi) & +& \nu\; \Phi(\varphi) & = &0 &&\text{for  $\varphi\in(0,2\pi)$} ,
    \end{array}    
    \end{equation*}
    subject to suitable boundary conditions.  Here, $(0,c)$ and $(0,-c)$ are the focal points of the ellipses and the hyperbolas in the elliptical coordinates. The resulting equations correspond to \emph{Mathieu's modified differential equation} and \emph{Mathieu's differential equation}, with the spectral parameters lacking separability.
    
    In higher dimensions, if $B=\diag(1/(b_{m+1}-b_1),\ldots,1/(b_{m+1}- b_m))$ is a diagonal matrix
    with $m$ different positive values $0<b_{m+1}-b_m<\ldots<b_{m+1}-b_1$ and the domain is the ellipsoid $\Omega=\{x\in \R^m\colon x^\T B x <1\}$,  separation can be achieved using ellipsoidal coordinates $\xi_1,\ldots,\xi_m$ defined by
    \[
    \sum_{\ell=1}^m\frac{x_\ell^2}{\xi_k-b_\ell}=1\quad \text{for $\xi_k\in(b_k,b_{k+1})$} .
    \]
    This leads to the $m$ coupled differential equations
    \begin{equation}\label{eq:ellipsoiderMEP}
        (p(\xi^{}_k)\, u_k'(\xi^{}_k))' +\frac{(-1)^{m-k}}{4p(\xi_k)}\sum_{\ell=1}^m \lambda^{}_\ell\;  \xi_k^{\ell-1} u^{}_k(\xi^{}_k) = 0\quad \text{for $\xi_k\in(b_k,b_{k+1})$},
    \end{equation}
    where $p(\xi)=\sqrt{ \prod_{\ell=1}^m\abs{b_\ell-\xi} }$ and appropriate boundary conditions are imposed;
    see e.g.,~\cite[\ch 6.9]{Volkmer88}, \cite[\sect 1.2]{SW1979}, or~\cite[\ch 1.13]{Meixner1954}.
    
    These problems are manifestations of multiparameter Sturm-Liouville eigenvalue problems. They take the general form
    \begin{equation}\label{eq:SturmMEP}
        (p^{}_k(x^{}_k)\, u_k'(x^{}_k))' +q_k^{}(x_k^{})\,u_k^{}(x_k^{})+\sum_{\ell=1}^m \lambda^{}_\ell\;  a^{}_{k\ell}(x^{}_k) u^{}_k(x_k^{}) = 0\quad \text{for $x_k\in(a_k,b_k)$}
    \end{equation}
    with appropriate boundary conditions. Standard assumptions are $p_k(x_k)>0$ for $x_k\in [a_k,b_k]$ and $p_k$ is continuously differentiable. Although this is not the case for~\eqref{eq:ellipsoiderMEP}, the positivity assumption can often be relaxed; see, for instance,~\cite[\ch 5.3]{Teschl2012}.

    Similar to the finite dimensional problem~\eqref{eq:MEP}, real eigenvalues $\lambda=(\lambda_1,\ldots,\lambda_m)$ have a multiindex since the occurring differential and multiplication operators are self-adjoint with respect to the inner product on $L^2(a_k,b_k)$ and the multiplication operators are bounded from above. There are according existence results for eigenvalues of signed index; see e.g.,~\cite[\theo{2.5.3} and \theo{2.7.1}]{Volkmer88}. Notably, if the eigenvalue problem is right definite, then there is an eigenvalue with signed index~$(\mathbf i,+)$. 
    Note that the differential problems have infinitely many eigenvalues, and can have both continuous and discrete spectrum.
    Right and left definiteness of~\eqref{eq:SturmMEP} can be verified pointwise. It is right definite if
    \[
        \det\begin{pmatrix}
            a_{11}(x_1)& \ldots &a_{1m}(x_1)\\
            \vdots&& \vdots\\
            a_{m1}(x_m)& \ldots &a_{mm}(x_m)
        \end{pmatrix}
        >0
    \] 
    on a dense subset of $[a_1,b_1]\times\ldots\times[a_m,b_m]$, and the analogue of~\eqref{eq:leftdeficond} for left definiteness with respect to $\mu$ is satisfied if
    \[
        \det\begin{pmatrix}
            a_{11}(x_1)&
            \ldots &
            a_{1m}(x_1)\\
            \vdots &  &\vdots\\
            a_{k-1,1}(x_{k-1})&
            \ldots &
            a_{k-1,m}(x_{k-1}) \\
            \mu_1 & \ldots  & \mu_m\\
            a_{k+1,1}(x_{k+1})&
            \ldots &
            a_{k+1,m}(x_{k+1})\\
            \vdots & &\vdots\\
            a_{m1}(x_m)&
            \ldots &
            a_{mm}(x_m)
        \end{pmatrix}
        >
        0
    \]
    on a dense subset of $[a_1,b_1]\times\ldots\times[a_m,b_m]$ 
    for $k=1,\ldots,m$; see e.g.,~\cite[\theo{3.6.2}]{Volkmer88}.
    
    The multiindex of an eigenvalue also has an interpretation as the number of internal zeros of the corresponding eigenfunctions $u_k$ of~\eqref{eq:SturmMEP}. An eigenfunction $u$ of the Sturm-Liouville eigenvalue problem
    \[
        (p(x)\,u'(x))' +q(x)\,u(x)=\lambda\, u(x)\quad\text{for $x\in(a,b)$}
    \]
    has $n$ internal zeros if $\lambda$ is the $n+1$-th largest eigenvalue under the general boundary conditions
    \[
        \cos(\alpha)\,u(a)+\sin(\alpha)\,u'(a)=0
        \quad \text{and}
        \quad\cos(\beta)\,u(b)+\sin(\beta)\,u'(b)=0,
    \]
    that include Dirichlet and Neumann boundary conditions. Under periodic boundary conditions it has 
    $2n$ internal zeros if $\lambda$ is the $2n$-th or $2n\!+\!1$-th largest eigenvalue. With antiperiodic boundary conditions it has $2n-1$ internal zeros if $\lambda$ is the $2n\!-\!1$-th or the $2n$-th largest eigenvalue; see e.g,~\cite[\theo 5.17 and \theo 5.37]{Teschl2012} and~\cite[\ch 8, \theo 2.1 and \theo 3.1]{Coddington55}.
    In the context of~\eqref{eq:SturmMEP}, a similar analogy holds for the multiindex of an eigenvalue. If $\lambda=(\lambda_1,\ldots,\lambda_m)$ has multiindex $\mathbf i=(i_1,\ldots,i_m)$ in the sense of \Cref{def:index}, then the corresponding eigenfunction $u_k$ has $i_k-1$ or $i_k$ internal zeros, contingent on the specified boundary conditions; see e.g.,~\cite[\theo 3.5.1]{Volkmer88}.   
    
    The techniques developed in \Cref{sec:NewtonMEP} are designed to calculate eigenvalues based on their multiindex. In the context of a discretization of a multiparameter Sturm-Liouville eigenvalue problem~\eqref{eq:SturmMEP}, the previous discussion leads to the additional interpretation: identifying an eigenvalue for which the associated eigenfunctions exhibit the designated number of interior zeros. Indeed, the discrete solutions to the finite differences analogue of Sturm-Liouville eigenvalue problem also have a similar property. For example, let the sequence $(u_k)_{k=0}^K$ solve the finite difference equation 
    \[
        a_k u_{k+1} - b_k u_k +a_{k-1}u_{k-1} = \lambda u_k
    \]
    with positive coefficients $a_k$ and boundary conditions $u_0= 0 = u_K$. Then the sequence $(u_k)_{k=0}^K$ has $n-1$ sign changes if $\lambda$ is the $n$-th largest solution; see e.g.,~\cite{Teschl_oscillation_1996}. Hence, solutions to the finite difference discretizations retain the oscillation property of the actual solutions.
    The approximation rates for eigenvalues are $\mathcal O(h^2)$ for uniform grids where $h$ is the width. The constants however depend on the smoothness of eigenfunctions, which are problem dependent and can be arbitrarily bad. When solutions are smooth other discretizations lead to better approximation rates, for example spectral collocation as in~\cite{Plestenjak2015}. Solutions to such discretizations can have more oscillations than the index of the eigenvalue. 

    \section{A Newton-type method}\label{sec:NewtonMEP}
    
    In what follows, let $\|\cdot\|$ be the Euclidean $\ell^2$-norm for vectors and the induced matrix norm for matrices which in this case is the spectral norm. We want to note however, that many arguments are independent of the respective norm, for example the result of \Cref{prop:minimumproperty}.

    Motivated by \Cref{def:index} and \Cref{thm:signedindex} we now define functions for each possible multiindex that equate to zero exactly when evaluated at a corresponding eigenvalue.
    For each multiindex $\mathbf{i}\in \{1,\ldots,n_1\}\times\ldots\times\{1,\ldots,n_m\}$, we define a function $F_\mathbf{i}\colon \R^m\to\R^m$ as follows
    \begin{equation}\label{eq:Newtonfunction}
        \lambda=(\lambda_1,\ldots,\lambda_m)\mapsto F_\mathbf{i}(\lambda)= (\varepsilon_{1,i_1}(\lambda),\ldots,\varepsilon_{m,i_m}(\lambda))  ,
    \end{equation}
    where $\varepsilon_{k,i_k}(\lambda)$ is the $i_k$-th largest eigenvalue of the matrix $\sum_{\ell=0}^m \lambda_\ell A_{k\ell}$ with $\lambda_0=1$, that is, the $k$-th component of $F_\mathbf i$ is the $i_k$-th largest eigenvalue of the matrix appearing in the $k$-th line of~\eqref{eq:MEP}. Therefore,
    the function $F_\mathbf{i}$ evaluates to zero if and only if $\lambda$ is an eigenvalue of the MEP~\eqref{eq:MEP} with the corresponding multiindex.
    
    Similarly, we define another function $\tilde F_\mathbf{i}\colon \S^m\to\R^m$ for the homogeneous MEP~\eqref{eq:HomMEP}
    \begin{equation}\label{eq:Newtonfunctionhom}
        \lambda=(\lambda_0,\ldots,\lambda_m)\mapsto \tilde F_\mathbf{i}(\lambda)= (\varepsilon_{1,i_1}(\lambda),\ldots,\varepsilon_{m,i_m}(\lambda)).
    \end{equation}
    If the MEP is locally definite, \Cref{thm:signedindex} establishes that $\tilde F_\mathbf{i}$ has precisely two zeros for every multiindex, one of positive sign and one of negative sign. In the case of right definiteness, $F_\mathbf{i}$ has a unique zero.
    The new task is to find the unique zero of the function $F_\mathbf i$ in the inhomogeneous case, or the unique zero in $\mathcal P^+$ or $P^-$ of $\tilde F_\mathbf i$ in the homogeneous case.

    Subsequently, we use Newton's method for the functions $F_\mathbf{i}$. To proceed with this, we need the Jacobian of $F_\mathbf{i}$, if it exists. Otherwise, we require its generalized Jacobian; see e.g.,~\cite{Clarke_Book_Opt} and~\cite[\ch{1.4.1}]{ISNewtonbook2014} for an overview.
    Clarke's generalized Jacobian of a Lipschitz-continuous function $\Phi\colon \R^n\to\R^m $ is given by
    \begin{equation}\label{eq: Clarke Jacobian}
       \partial \Phi(x) = \conv \{J\in \R^{m\times n}\colon \text{there is a sequence }(x_k)_k\subset \mathcal S_\Phi \text{ s.t. }\lim_{k\to\infty} x_k = x\text{ and } \lim_{k\to\infty} \Phi'(x_k) = J\},
    \end{equation}
    where $\mathcal S_\Phi$ is the set of points where $\Phi$ is differentiable.
    Let us collect some required properties of the Jacobian of $F_\mathbf i$ for the application of Newton's method.
    For this, we require the following statement about open sets.
    \begin{lemma}\label{lem: openSetlemma}
        Let $(M,d)$ be a metric space, and $A, B\subset M$ be sets such that $A$ is open and $A\cup B$ is open. Then $A\cup B^\circ $ is dense in $A\cup B$, where $B^\circ$ denotes the interior of $B$.
    \end{lemma}
    \begin{proof}
       Let $x\in B$ and $\varepsilon>0$ be small enough such that $d(x,y)<\varepsilon$ implies $y\in A\cup B$. Now either $d(x,y)<\varepsilon$ implies $y\in B$ or there is $y\in A$ with $d(x,y)<\varepsilon$. In the first case $x\in B^\circ$. If the second case is true for all small enough $\varepsilon>0$, then there is a sequence $(x_k)_k\subset A$ with $\lim_{k\to \infty} x_k = x$. This proves the claim.
    \end{proof}
    \begin{proposition}
        The function $F_\mathbf i$ is differentiable at a point $\lambda$ if the eigenvalues $\varepsilon_{k,i_k}(\lambda)$ have constant multiplicity in a neighborhood of $\lambda$. Then its 
        Jacobian is given by 
        \begin{equation}\label{eq:differentialF}
            F'_{\mathbf i}(\lambda)=
            \begin{pmatrix}
            u_1^\He A^{}_{11} u^{}_1&\ldots&u_1^\He A^{}_{1m} u^{}_1\\
            \vdots&&\vdots\\
            u_m^\He A^{}_{m1} u^{}_m&\ldots&u_m^\He A^{}_{mm} u^{}_m
            \end{pmatrix},
        \end{equation}
        where the vectors $u_k\in\mathcal U_k$ are eigenvectors of $\sum_{\ell=0}^m \lambda_\ell A_{k\ell}$ corresponding to the eigenvalue $\varepsilon_{k,i_k}(\lambda)$. Otherwise, its generalized Jacobian is contained in the convex hull of
        \begin{equation}\label{eq: generalized Jacobian}
            \left\{\!\!\begin{pmatrix}
                u_1^\He A^{}_{11} u^{}_1&
                \!\!\!\ldots\!\!\!
                &u_1^\He A^{}_{1m} u^{}_1\\
                \vdots&&\vdots\\
                u_m^\He A^{}_{m1} u^{}_m&
                \!\!\!\ldots\!\!\!
                &u_m^\He A^{}_{mm} u^{}_m
                \end{pmatrix}
                \!
                \colon
                \text{$u_k\in\mathcal U_k$ is an eigenvector corresponding to $\varepsilon_{k,i_k}(\lambda)$}
                \!\right\}.
        \end{equation}
    \end{proposition}
    
    \begin{proof}
        Let $A\colon I\subset \R\to \R^{n\times n}$ be a matrix valued function. Let $\lambda(t)$ be an eigenvalue of $A(t)$ smoothly depending on $t$, and $u(t)$, $w(t)$ be smooth corresponding right and left eigenvectors. Then 
        \[
        A'(t) u(t) +A(t) u'(t) = \frac{d}{dt}(A(t) u(t))  =
        \frac{d}{dt}(\lambda(t) u(t) ) = \lambda'(t) u(t) +\lambda(t) u'(t).
        \]
        Left multiplication by $w^\He(t)$ leads to 
        \[
            w^\He(t)A'(t) u(t) = \lambda'(t) w^\He(t) u(t).
        \]
        As a result, we have 
        \[
        \frac{\partial}{\partial \lambda_\ell} \varepsilon_{k,i_k}(\lambda) = u_k^\He A_{k\ell} u_k
        \]
        with $u_k$ being a normalized eigenvector corresponding to $\varepsilon_{k,i_k}(\lambda)$
        if the eigenvalue $\varepsilon_{k,i_k}(\lambda)$ has a partial differential in direction $\lambda_\ell$. Here, we also made use of the fact that left and right eigenvectors are the same for Hermitian matrices.

        Indeed, by~\cite[\ch 2, \theo 6.1]{Kato1976} eigenvalues and eigenvectors of a holomorphic family of Hermitian matrices are analytic when regarding them as an unordered set. Therefore, when multiplicity of an eigenvalue does not change in a neighborhood both the eigenvalue and its eigenspace is holomorphic. Hence, all the partial derivatives 
        $\frac{\partial}{\partial \lambda_\ell} \varepsilon_{k,i_k}(\lambda)$ exist and are continuous when the multiplicity of the eigenvalue is constant in a neighborhood of $\lambda$. Thus,~\eqref{eq:differentialF} follows.

    For the second claim, notice that the multiplicity of an eigenvalue is upper semicontinuous, that is, can only jump upwards. For $m \in \N_0$ we define the sets of constant multiplicity
    \[
        S^{m}_k = \{\lambda\colon \varepsilon_{k,i_k}(\lambda) \text{ has multiplicity $m$}\}.
    \]
    By upper semicontinuity, the union
    \[
        \bigcup_{m'\leq m }S^{m'}_k= \{\lambda\colon \varepsilon_{k,i_k}(\lambda) \text{ has multiplicity at most $m$}\}
    \]
    is open.
    Using \Cref{lem: openSetlemma} iteratively, we get that $\bigcup_{m'\in \N_0  }(S^{m'}_k)^\circ$ is a dense open set, hence by~\cite[\theo{2.5.1}]{Clarke_Book_Opt}
    \begin{equation*}
        \partial \varepsilon_{k,i_k}(\lambda) = \conv \{J\in \R^{1\times m}\colon \text{there is a seq. }(\lambda_j)_j\subset \bigcup_{m'\in \N_0 }(S^{m'}_k)^\circ  \text{ s.t. }\lim_{j\to\infty} \lambda_j = \lambda\text{ and } \lim_{j\to\infty} \varepsilon_{k,i_k}'(\lambda_j) = J\}
     \end{equation*}
    that is, the sequences in the definition~\eqref{eq: generalized Jacobian} for the generalized gradient can be restricted to a subset with full measure. It follows that
    \[
        \partial \varepsilon_{k,i_k}(\lambda)\subseteq \conv\{\begin{pmatrix}
            u_k^\He A^{}_{k1} u^{}_k&
            \!\!\!\ldots\!\!\!
            &u_k^\He A^{}_{km} u^{}_k
            \end{pmatrix}
            \!
            \colon
            \text{$u_k\in\mathcal U_k$ is an eigenvector corresponding to $\varepsilon_{k,i_k}(\lambda)$}
            \}
    \] 
    as for $\lambda_j\in \bigcup_{m'\in \N_0  }(S^{m'}_k)^\circ$ we have $\varepsilon_{k,i_k}(\lambda_j) = \begin{pmatrix}
        (u^j_k)^\He A^{}_{k1} u^{j}_k&
        \!\!\!\ldots\!\!\!
        &(u^j_k)^\He A^{}_{km} u^{j}_k
        \end{pmatrix}$
        with a corresponding eigenvector $u_k^j\in \mathcal U_k$ and by continuity of eigenvectors we have that limit points of $ u_k^j$ are again eigenvectors corresponding to $\varepsilon_{k,i_k}(\lambda)$. The claim follows by 
        \[
        \partial F_\mathbf i (\lambda) \subseteq \partial \varepsilon_{1,i_1}(\lambda)\times \ldots \times \partial \varepsilon_{m,i_m}(\lambda);
        \]
        see e.g.,~\cite[\pr{1.54}]{ISNewtonbook2014}.
    \end{proof}
    
    Even though the functions $F_\mathbf i$ are not necessarily everywhere differentiable, we can still employ a semismooth Newton method. The iteration scheme consists of two steps. First, any element $J$ of the generalized Jacobian $\partial F_\mathbf i(\lambda^{(j)})$ needs to be computed. Then the next iterate is given by the solution of the linear equation $F_\mathbf i(\lambda^{(j)}) + J(\lambda^{(j+1)}- \lambda^{(j)}) = 0$.
    
    A specific choice of~$J\in\partial F_\mathbf i\left(\lambda^{(j)}\right)$  with eigenvectors~$u_k\in\mathcal U_k$ corresponding to $\varepsilon_{k,i_k}\left(\lambda^{(j)}\right)$ can be computed using~\eqref{eq: generalized Jacobian}. Then $F_\mathbf i\left(\lambda^{(j)}\right)$ is given by
    \[
        F_\mathbf i\left(\lambda^{(j)}\right)
        =
        W(u_1,\ldots,u_m)
        \begin{pmatrix}
        1\\
        \lambda^{(j)}
        \end{pmatrix}
        =
        \begin{pmatrix}
        u_1^\He A^{}_{10}u_1^{} & \ldots  & u_1^\He A^{}_{1m}u_1^{} \\
        \vdots &  & \vdots                         \\
        u_m^\He A^{}_{m0}u_m^{} & \ldots  & u_m^\He A^{}_{mm}u_m^{}
        \end{pmatrix} 
         \begin{pmatrix}
        1\\
        \lambda^{(j)}_1\\
        \vdots\\
        \lambda^{(j)}_m
        \end{pmatrix}. 
    \]
    It follows, that the subsequent iterate $\lambda^{(j+1)}$ is determined by the solution to the linear equation
    \[
    W(u_1,\ldots,u_m)
    \begin{pmatrix}
    1\\
    \lambda^{(j+1)}
    \end{pmatrix}=0.
    \]
    If the MEP~\eqref{eq:MEP} is right definite, this equation possesses a unique solution. This procedure is summarized in \Cref{alg:Newton}.
    It can be adapted to handle locally definite MEPs by ensuring that iterates remain within the sets~$\mathcal P^\sigma$. This adaptation is outlined in \Cref{alg:NewtonHom}.
    It is also noteworthy that $\lambda^{(j+1)}$ can be seen as a generalization of the Rayleigh quotient corresponding to the vectors $u_1, \ldots, u_m$ and coincides with the tensor Rayleigh quotient described in e.g.~\cite{HP2002,Plestenjak2000,HMMP2019}. For algorithmic purposes we start the method with approximations of eigenvectors which are in turn used to compute eigenvalue approximations by this generalization of the Rayleigh quotient. This ensures that the first iterate~$\lambda^{(1)}$ is in~$\mathcal Q$ and~$\mathcal P^+$, respectively.
    Depending on the method used to compute eigenvalues the eigenvector approximations can be utilized. The iteration steps end by the computation of eigenvalue problems. The reason behind this is that possible stopping criteria for the algorithm rely on the value of $F_\mathbf i(\lambda^{(j)})$ which consists precisely of the resulting eigenvalues.

    To provide a rate of convergence, it is crucial that the function of interest is semismooth. Indeed, the functions $\varepsilon_{k,i_k}$ are even strongly semismooth, as shown in~\cite{SunSun2002}, meaning that 
    \begin{equation}\label{eq:semismooth}
        \sup_{J\in\partial \varepsilon_{k,i_k}(\lambda+\Delta\lambda)}
        \|\varepsilon_{k,i_k}(\lambda+\Delta\lambda)-\varepsilon_{k,i_k}(\lambda)+J \Delta\lambda\|\in
        \mathcal O(\|\Delta\lambda\|^2),
    \end{equation}
    where $\Delta\lambda$ denotes a small displacement.
    When a function from $\R^n$ into $\R^n$ is semismooth, a semismooth Newton method can be applied. If the function is strongly semismooth, local quadratic convergence is retained; see e.g.,~\cite[\ch 2]{ISNewtonbook2014} for an overview.
       
    \begin{algorithm}[t]
    \caption{Semismooth Newton method for right definite MEPs}\label{alg:Newton}
    \begin{algorithmic}
    \Require A right definite MEP of the form~\eqref{eq:MEP} and a multiindex $\mathbf i$
    \Ensure Approximation of an eigenvalue $\lambda$ of multiindex $\mathbf i$ and corresponding eigenvector $u_1,\ldots,u_m$
    \State initial guess of eigenvector $u^{(0)}_1,\ldots,u^{(0)}_m$
    
    \For {$j=1,2,\dots$}
    \State
    solve $
    W(u^{(j-1)}_1,\ldots,u^{(j-1)}_m)
    \begin{pmatrix}
    1\\
    \lambda^{(j)}
    \end{pmatrix}=0$
        \For {$k=1,\dots,m$}
            \State compute the $i_k$-th largest eigenvalue of $\sum_{\ell=0}^m \lambda^{(j)}_\ell A^{}_{k\ell}$ 
            \State compute a corresponding eigenvector $u^{(j)}_k\in\mathcal U_k$
        \EndFor
        
    \EndFor
    \State
    \Return $\lambda^{(j)}, u^{(j)}_1,\ldots,u^{(j)}_m$
    \end{algorithmic}
    \end{algorithm}

    \begin{algorithm}[t]
        \caption{Semismooth Newton method for locally definite MEPs}\label{alg:NewtonHom}
        \begin{algorithmic}
        \Require A locally definite MEP of the form~\eqref{eq:HomMEP} and a multiindex $\mathbf i$
        \Ensure Approximation of an eigenvalue $\lambda$ of signed index $(\mathbf i,\sigma)$  and corresponding eigenvector $u_1,\ldots,u_m$
        \State initial guess of eigenvector $u^{(0)}_1,\ldots,u^{(0)}_m$
        \For {$j=1,2,\dots$}
        \State
        solve $
        W(u^{(j-1)}_1,\ldots,u^{(j-1)}_m)
        \lambda^{(j)}
        =0$
        with $\lambda\in \S^n$ such that $\sign\det \begin{pmatrix}
            (\lambda^{(j)})^\T\\
            W(u_1,\ldots,u_m)
        \end{pmatrix} = \sigma$
            \For {$k=1,\dots,m$}
                \State compute the $i_k$-th largest eigenvalue of $\sum_{\ell=0}^m \lambda^{(j)}_\ell A^{}_{k\ell}$ 
                \State compute a corresponding eigenvector $u^{(j)}_k\in\mathcal U_k$
            \EndFor

        \EndFor
        \State
        \Return $\lambda^{(j)}, u^{(j)}_1,\ldots,u^{(j)}_m$
        \end{algorithmic}
        \end{algorithm}

    Since the functions $\varepsilon_{k,i_k}$ are strongly semismooth, $F_\mathbf{i}$ is also strongly semismooth; refer to e.g.,~\cite[\pr 1.73]{ISNewtonbook2014}.  As a result, applying a semismooth Newton method to $F_\mathbf{i}$ achieves local quadratic convergence if all matrices in the generalized Jacobians at solutions are nonsingular; see e.g.,~\cite[\theo 2.42]{ISNewtonbook2014}.
    To demonstrate this, we establish that the generalized Jacobian~\eqref{eq: generalized Jacobian} possesses a uniformly bounded inverse.

    \begin{proposition}\label{prop:gradientcondition}
        Let the MEP~\eqref{eq:MEP} be right definite and let 
        \[
            \mathcal J 
            \coloneqq
            \left\{
            \begin{pmatrix}
                u_1^\He A^{}_{11} u^{}_1&
                \ldots
                &u_1^\He A^{}_{1m} u^{}_1\\
                \vdots&&\vdots\\
                u_m^\He A^{}_{m1} u^{}_m&
                \ldots
                &u_m^\He A^{}_{mm} u^{}_m
                \end{pmatrix}
                \colon
                \text{$u_k\in\mathcal U_k$ for $k=1,\ldots,m$}
            \right\}.
        \]
        Then there exists $\gamma<\infty$ such that $\left\|J^{-1}\right\|\leq \gamma$ for all $J\in \conv\mathcal J$.
    \end{proposition}
    \begin{proof}
        First note that $\mathcal J$ is a compact set in $\R^{m\times m}$ as it is the image of the continuous map
        \[
            J\colon \mathcal U_1\times\ldots\times \mathcal U_m \to \R^{m\times m},\quad u=(u_1,\ldots,u_d)\mapsto J(u)\!=\! \begin{pmatrix}
                u_1^\He A^{}_{11} u^{}_1&
                \!\ldots\!
                &u_1^\He A^{}_{1m} u^{}_1\\
                \vdots&&\vdots\\
                u_m^\He A^{}_{m1} u^{}_m&
                \!\ldots\!
                &u_m^\He A^{}_{mm} u^{}_m
            \end{pmatrix}
        \]
        of the compact set $\mathcal U_1\times\ldots\times \mathcal U_m$. It follows from Carath\'eodory's theorem that $\conv \mathcal J$ is also compact since $\R^{m\times m}$ is finite dimensional; see e.g.,~\cite[\theo 1.1 and \coro 1.9]{Tuy2016}. Finally, we show that $\det J\neq 0$ for all $J\in \conv \mathcal J$. 
        Then $\conv \mathcal J$ is a compact subset of invertible matrices and therefore the set $\{J^{-1}\colon J\in\conv \mathcal J\}$ is compact as inverting a matrix is a continuous map on invertible matrices. This implies the existence of the desired $\gamma<\infty$.
        
        We proceed by induction. 
        Let $\mathcal J_k$ be the set consisting of the $k$-th row of matrices in $\mathcal J$, i.e.,
        \[
            \mathcal J_k=\{\begin{pmatrix}
            u_k^\He A^{}_{k1} u^{}_k&
            \ldots
            &u_k^\He A^{}_{km} u^{}_k
            \end{pmatrix}\colon u_k\in \mathcal U_k\}.
        \]
        Furthermore, let $\mathcal V_k$ be the set of matrices where the rows $\ell=1,\ldots,k$ consist of elements in the convex sets $\conv \mathcal J_\ell$, respectively, and the other rows consist of elements in $\mathcal J_\ell$ for $\ell=k+1,\ldots, m$, i.e.,
        \[
            \mathcal V_k=\left\{ 
            \begin{pmatrix}
                J_1\\
                \vdots\\
                J_m
            \end{pmatrix}
            \colon
            \text{$J_\ell\in \conv\mathcal J_\ell$ for $\ell=1,\ldots, k$ and $J_\ell\in \mathcal J_\ell$ for $\ell=k+1,\ldots, m$}
            \right\}.
        \]
        From right definiteness, it follows that $\det J>0$ for $J\in\mathcal J=\mathcal V_0$. 
        Now let $\det J>0$ for all $J\in \mathcal V_{k-1}$. Then any $J\in \mathcal V_k$ is of the form $J=\sum_{j=1}^M \sigma_j J^{(j)}$ where $\sigma_j\geq 0$, $\sum_{j=1}^M\sigma_j=1$, and
        \[
            J^{(j)}
            =
            \begin{pmatrix}
                J_1\\
                \vdots\\
                J_{k-1}\\
                J^{(j)}_k\\
                J_{k+1}\\
                \vdots\\
                J_m
            \end{pmatrix}
            \quad 
            \text{for}\quad \left\{
            \begin{array}{l}
                J_\ell \in \conv\mathcal J_{\ell}\text{ for $\ell=1,\ldots, k-1$},\\
                J_k^{(j)}\in \mathcal J^{}_k,\\
                J_\ell \in \mathcal J_{\ell} \text{ for $\ell=k+1,\ldots, m$}.
            \end{array}\right.
        \]
        Then by linearity of the determinant in its rows, we have $\det J=\sum_{j=1}^M \sigma_j \det J^{(j)}>0$ as $J^{(j)}\in \mathcal V_{k-1}$. It follows by induction that $\det J>0$ for all $J\in \mathcal V_k$ and $k=0,\ldots, m$.

        Let $J\in \conv \mathcal J$. By definition, $J=\sum_{j=1}^M \sigma_j J^{(j)}$ for some $J^{(j)}\in \mathcal J$, $\sum_{j=1}^M \sigma_j =1$, and $\sigma_j\geq 0$. Then 
        \[
            J
            =
            \begin{pmatrix}
                \sum_{j=1}^M \sigma^{}_j J^{(j)}_1\\
                \vdots\\
                \sum_{j=1}^M \sigma^{}_j J^{(j)}_m
            \end{pmatrix}
            \quad 
            \text{for some $J_k^{(j)}\in \mathcal J_k$, $k=1,\ldots m $.}
        \]
        It follows that $\conv\mathcal J\subset \mathcal V_m$ and the assertion is proven.
        \end{proof}

        The convergence results are inherently local. It is possible to employ globalization strategies involving a line search, but quadratic convergence might be sacrificed. In many cases, our numerical experiments in \Cref{sec:numericalMEP} did not require globalization. We also only use the simple globalization strategy in \Cref{alg:NewtonGlobalized}.
        In the following section, we derive a global convergence result for certain eigenvalues in \Cref{thm:globalconvergence}.

        In the proposed Newton methods a good initialization can be obtained when an approximation of the desired eigenvector is available. For multiparameter Sturm-Liouville problems such approximations can be obtained from other eigenvectors corresponding to other multiindices. This strategy is for example used in~\cite{GHP_Math_2012} and was also tried for the method in~\cite{eisenmann2021solving} which is very similar to the Newton method of this work.

        \begin{algorithm}[t]
            \caption{Globalized semismooth Newton method for right definite MEPs}\label{alg:NewtonGlobalized}
            \begin{algorithmic}
                \Require A right definite MEP of the form~\eqref{eq:MEP} and a multiindex $\mathbf i$
                \Ensure Approximation of an eigenvalue $\lambda$ of multiindex $\mathbf i$ and corresponding eigenvector $u_1,\ldots,u_m$
                \State initial guess of eigenvector $u^{(0)}_1,\ldots,u^{(0)}_m$ and stepwidth parameter $\tau<1$
                \For {$j=1,2,\dots$}
                \State
                solve $
                W(u^{(j-1)}_1,\ldots,u^{(j-1)}_m)
                \begin{pmatrix}
                1\\
                \lambda^{(j)}
                \end{pmatrix}=0$
                \Repeat
                    \For {$k=1,\dots,m$}
                        \State compute the $i_k$-th largest eigenvalue $\varepsilon^{(j)}_k$ of $\sum_{\ell=0}^m \lambda^{(j)}_\ell A^{}_{k\ell}$ 
                        \State compute a corresponding eigenvector $u^{(j)}_k\in\mathcal U_k$
                    \EndFor
                \If{$j>1$ and $\|(\varepsilon_1^{(j-1)}, \dots, \varepsilon_m^{(j-1)})\|_\infty < \|(\varepsilon_1^{(j)}, \dots, \varepsilon_m^{(j)})\|_\infty$}
                \State  $\lambda^{(j)} \leftarrow \tau \lambda^{(j)} +(1-\tau)\lambda^{(j-1)}$
                \EndIf

                \Until{$\|(\varepsilon_1^{(j-1)}, \dots, \varepsilon_m^{(j-1)})\|_\infty > \|(\varepsilon_1^{(j)}, \dots, \varepsilon_m^{(j)})\|_\infty$}
                \EndFor
                \State
                \Return $\lambda^{(j)}, u^{(j)}_1,\ldots,u^{(j)}_m$
                \end{algorithmic}
            \end{algorithm}
    \subsection{Extreme eigenvalues}\label{sec:OptMEP}

    The task of identifying an eigenvalue with an extreme multiindex $\mathbf i \in \{1,n_1\}\times\ldots\times \{1,n_m\}$ can be approached from an optimization perspective. We again note that eigenvalues of the MEP~\eqref{eq:MEP} lie in the set
    \[
    \mathcal Q = \left\{\lambda\in\R^m\colon W(u)\begin{pmatrix}
       1\\
       \lambda
    \end{pmatrix}=0 \text{ for some $u\in \mathcal U_1\times\ldots\times \mathcal U_m$}  \right\}.
    \]
    The set $\mathcal Q$ is in general not convex, but eigenvalues with extreme indices serve as extreme points within the convex hull of $\mathcal Q$.
    
    \begin{proposition}
    Let the MEP~\eqref{eq:MEP} be right definite. Then 
    \[
        \mathcal Q\subseteq \conv\left\{\lambda^{\mathbf i}\colon  \mathbf i\in\{1,n_1\}\times\ldots\times \{1,n_m\}\right\}
    \]
    where $\lambda^{\mathbf i}$ denotes the eigenvalue of multiindex $\mathbf i$.
    \end{proposition}
    \begin{proof}
    Let $\lambda\in \mathcal Q$ and let $u\in \mathcal U_1\times\ldots\times \mathcal U_m$ be a corresponding vector such that $W(u)\begin{pmatrix}
       1\\
       \lambda
    \end{pmatrix}=0$. We denote $\mathcal E=\{1,n_1\}\times\ldots\times \{1,n_m\}$ as the set of extreme multiindices. Next, notice that 
    \[
    W(u)\begin{pmatrix}
       1\\
       \lambda^{\mathbf i}
       \end{pmatrix}
       =
       \begin{pmatrix}
          \epsilon_1\\
          \vdots\\
          \epsilon_m
       \end{pmatrix}
    \]
    where $\epsilon_k\geq0$ if $i_k=n_k$ and $\epsilon_k\leq0$ if $i_k=1$ as $0$ is either the smallest or largest eigenvalue of $\sum_{\ell=0}^m \lambda^\mathbf i_\ell A_{kl}$ by the definition of $\lambda^\mathbf i$. Hence,
    \[
    0\in \conv\left\{W(u)\begin{pmatrix}
       1\\
       \lambda^{\mathbf i}
       \end{pmatrix}\colon \mathbf i\in \mathcal E \right\}
       =
        W(u)\left\{\begin{pmatrix}
       1\\
       \lambda
       \end{pmatrix}\colon \lambda \in \conv \left\{\lambda^{\mathbf i}\colon\mathbf i\in \mathcal E\right\} \right\}
    \]
    as for each $\mathbf i\in \mathcal E$ the vectors 
    $ W(u)\begin{pmatrix} 1\\
    \lambda\end{pmatrix}$
    lie in different quadrants of $\R^m$.
    Since $W(u)$ is of full rank, the solution of $W(u)\begin{pmatrix}
       1\\
       \lambda
    \end{pmatrix}=0$ is unique and the assertion follows.
    \end{proof}

    Directly from the previous discussion, we can observe that
    \[
    \min_{\lambda\in\mathcal Q}\mu^\T\lambda
    \]
    is achieved at an eigenvalue associated with an extreme multiindex. In the case where the MEP~\eqref{eq:MEP} is both left definite with respect to $\mu$ and right definite, we have the ability to identify which extreme multiindex corresponds to a minimizing eigenvalue.
    \begin{proposition}\label{prop:minimumproperty}
    Let the MEP~\eqref{eq:MEP} be right definite and left definite with respect to $\mu$. Then 
    \[
    \min_{\lambda\in\mathcal Q}\mu^\T\lambda
    \]
    is attained at the eigenvalue  $\lambda^{\mathbf 1}$ of multiindex  $\mathbf{1}=(1,\ldots, 1)$. Furthermore, there is a $\beta >0$ such that $\mu^\T\left(\lambda-\lambda^{\mathbf 1}\right)\geq \beta \left\|\lambda-\lambda^{\mathbf 1}\right\|$ for all $\lambda\in \mathcal Q$.
    \end{proposition}
    \begin{proof}
    Let $\lambda\in\mathcal Q$ and choose $u\in\mathcal U_1\times\dots\times \mathcal U_m$ such that 
    $W(u)
    \begin{pmatrix}
       1\\ \lambda
    \end{pmatrix}
    =0$. Note that the entries of  $W(u)
    \begin{pmatrix}
       1\\ \lambda^{\mathbf 1}
    \end{pmatrix}
    $ are nonpositive since $0$ is the largest eigenvalue of the matrices $\sum_{\ell=0}^m \lambda_{\ell}^{\mathbf 1}A^{}_{k\ell}$ for $k=1,\ldots, m$. Hence, for the matrix $J(u)$ defined via the equality $J(u) \lambda =  W(u)\begin{pmatrix} 0 \\ \lambda\end{pmatrix}$, we have 
    \begin{equation}\label{eq:defineJmatrix}
        J(u)\left(\lambda - \lambda^{\mathbf 1}\right) =
    \begin{pmatrix}
       u_1^\He A_{11}^{}u_1^{}&\ldots &u_1^\He A_{1m}^{}u_1^{}\\
       \vdots &&\vdots\\
       u_m^\He A_{m1}^{}u_m^{}&\ldots &u_m^\He A_{mm}^{}u_m^{}
    \end{pmatrix}
    \left(\lambda - \lambda^{\mathbf 1}\right) = 
    \begin{pmatrix}
      \epsilon_1\\
      \vdots\\
      \epsilon_m
    \end{pmatrix}
x    \end{equation}
    with $\epsilon_k\geq0$ for $k=1,\ldots, m$. If $\lambda\neq \lambda^{\mathbf 1}$, then at least one inequality is strict and by \Cref{prop:gradientcondition} there is~$\tilde \gamma$
    \[
    \tilde\gamma \max_{k=1,\ldots,m} \abs{\epsilon_k}
    = \tilde \gamma \| J(u) (\lambda) - \lambda^\mathbf 1 \|_\infty  
    \geq \left\|\lambda - \lambda^\mathbf 1 \right\|.
    \] By right definiteness the matrix $J(u)$ is invertible, and by Cramer's rule and the condition~\eqref{eq:leftdeficond} for left definiteness the entries of $\mu^\T (J(u))^{-1}$ are positive. Furthermore, by compactness of the sets $\mathcal U_1,\ldots, \mathcal U_m$ there is an $\alpha>0$ such that the entries are at least $\alpha$ for every $u\in\mathcal U_1\times\ldots\times \mathcal U_m$.
    Hence,
    \[
    \mu^\T\left(\lambda^{\mathbf 1}-\lambda\right)
    =
    \mu^\T (J(u))^{-1}J(u)\left(\lambda-\lambda^{\mathbf 1}\right)
    \geq \alpha  \| J(u) (\lambda) - \lambda^\mathbf 1 \|_\infty  
    \geq \frac{\alpha}{\tilde \gamma}\left\|\lambda- \lambda^{\mathbf 1}\right\|,
    \]
    which shows the assertion.
    \end{proof}
    
    Similarly, we can argue that $\mu^\T\left(\lambda^{\mathbf i}-\lambda^{\mathbf j}\right)\leq 0$ holds whenever $i_k\leq j_k$ for $k=1,\ldots,m$. This can also be formulated in the following way.
    Let $\mathbf i\leq \mathbf j$ whenever $i_k\leq j_k$ for all $k = 1,\ldots, m$, i.e., the product partial order, and $\mathbf i< \mathbf j$ if $\mathbf i\leq \mathbf j$ and $\mathbf i\neq \mathbf j$.
    The mapping $\mathbf i\mapsto \mu^\T \lambda^\mathbf{i}$ is order-preserving concerning the product order on multiindices, that is, $\mu^\T\lambda^{\mathbf i} 
    \leq
    \mu^\T\lambda^{\mathbf j}$ when $\mathbf i \leq \mathbf j$. This is for example the case in multiparameter Sturm-Liouville problems coming from the Helmholtz equation \eqref{eq:helmholtz}. Then the eigenvalue of the original multidimensional Laplace operator can only increase if one component function has a larger number of internal zeros.
    
    \begin{proposition}\label{prop:smallindices}
    Let the MEP~\eqref{eq:MEP} be right definite and left definite with respect to $\mu$. Let $\lambda^{\mathbf i}$ and $\lambda^{\mathbf j}$ be eigenvalues of multiindices $\mathbf i$ and $\mathbf{j}$. Then 
    \[
    \mu^\T\lambda^{\mathbf i} 
    \leq
    \mu^\T\lambda^{\mathbf j}
    \]
    if $i_k\leq j_k$ for $k=1,\ldots, m$. 
    \end{proposition}
    \begin{proof}
    We show that there is a $u\in\mathcal U_1\times\dots\times \mathcal U_m$ such that 
    $J(u)\left(\lambda^{\mathbf i}-\lambda^{\mathbf j}\right)$ has nonpositive entries. Then the result follows analogously as in the proof of \Cref{prop:minimumproperty}. 
    Since $0$ is the $j_k$-th largest eigenvalue of 
    $\sum_{\ell=0}^m \lambda_{\ell}^{\mathbf j}A^{}_{k\ell}$ and $0$ is the $i_k$-th largest eigenvalue of
    $\sum_{\ell=0}^m \lambda_{\ell}^{\mathbf i}A^{}_{k\ell}$, and since $i_k\leq j_k$, there is a $u_k\in \mathcal U_k$ such that 
    $u_k^\He\sum_{\ell=0}^m \left(\lambda_{\ell}^{\mathbf i}-\lambda_{\ell}^{\mathbf j}\right)A^{}_{k\ell}u_k^{}\leq 0$. 
    It follows that $J(u_1,\ldots,u_m)\left(\lambda_{\ell}^{\mathbf i}-\lambda_{\ell}^{\mathbf j}\right)$ has nonpositive entries and the result follows.
    \end{proof}
    
    Therefore, in the case of a MEP~\eqref{eq:MEP} being left and right definite, it is possible to compute eigenvalues with small $\mu^\T \lambda$ using \Cref{alg:Newton} by targeting small multiindices. Additionally, left definiteness ensures global convergence of the methods when the desired eigenvalue has the multiindex $\mathbf 1$.

\begin{theorem}\label{thm:globalconvergence}
    Let the MEP~\eqref{eq:MEP} be left and right definite.
    The sequence $\lambda^{(j)}$ generated by \Cref{alg:Newton} converges globally to $\lambda^{\mathbf 1}$ when the required multiindex is $\mathbf 1$.
\end{theorem}
    \begin{proof}
    First notice that $\lambda^{(j)}\in\mathcal Q$ for $j\geq1$. Due to \Cref{prop:minimumproperty} it is enough to show that $\mu^\T \lambda^{(j)}$ converges to $\mu^\T \lambda^{\mathbf 1}$. From \Cref{prop:gradientcondition} and the characterization of the directional derivatives of $F_{\mathbf 1}$, it follows that
    \[
    \gamma \|F_{\mathbf 1}(\lambda)\| \geq \left\|\lambda-\lambda^{\mathbf 1}\right\|,
    \]
    and furthermore the entries of $F_{\mathbf 1}$ are nonnegative by the definition of $\mathcal Q$ and $F_{\mathbf 1}(\lambda)$. It follows, that the largest entry of $F_\mathbf 1(\lambda)$ is at least $\delta\mu^\T (\lambda - \lambda^{\mathbf 1})$ with some $\delta >0$ independent of $\lambda$ as $\mu^\T$ acts as a bounded linear functional.  The difference $\lambda^{(j)}-\lambda^{(j+1)}$ satisfies
    \[
    J\left(\lambda^{(j)}-\lambda^{(j+1)}\right)=F_\mathbf 1 \left(\lambda^{(j)}\right)
    \]
    for some $J\in\conv \mathcal J$ of \Cref{prop:gradientcondition} by the iteration formula of Newton's method. With a similar argument as in the proof of \Cref{prop:gradientcondition}, we get that $\mu^\T J^{-1}$ has entries larger than some $\alpha$ for every $J\in \conv\mathcal J$. Hence,
    \[
    \mu^\T\left(\lambda^{(j)}-\lambda^{(j+1)}\right)\geq
    \alpha \delta
    \mu^\T\left( \lambda^{(j)}-\lambda^{\mathbf 1}\right)
    \]
    and thus
    \[
    \mu^\T \left(\lambda^{(j+1)}-\lambda^{\mathbf 1}\right)
    \leq 
    \left(1-\alpha \delta \right)
    \mu^\T\left(\lambda^{(j)}-\lambda^{\mathbf 1}\right),
    \]
    that is global linear convergence of the sequence $\mu^\T\lambda^{(j)}$ to $\mu^\T\lambda^{\mathbf 1}$.
    \end{proof}

    \subsection{Handling non-Hermitian matrices}\label{sec: non Hermitian}

    The Newton method \Cref{alg:Newton} can be applied to multiparameter eigenvalue problems involving non-Hermitian matrices for the purpose of seeking eigenvalues associated with specific multiindices. Provided that all eigenvalues are real, \Cref{def:index} can be extended to non-Hermitian matrices. The challenge lies in determining precisely which multiindices are realized in the solution, particularly when dealing with non-Hermitian matrices. The situation becomes more straightforward when it is a discretization of a selfadjoint problem. However, complications may arise when discretization or transformations compromise this selfadjoint property. For this purpose, we consider the scenario, when  there exist diagonal matrices $D_1^L, \ldots, D_m^L $ and $D_1^R, \ldots, D_m^R$ with positive entries such that
    \begin{equation}\label{eq: transformed MEP}
        \setlength\arraycolsep{1pt}
        \begin{array}{cccccccccc}
            D_1^L\bigl(A_{10}& +&\lambda_1 A_{11}
            &+&\ldots&+& \lambda_m A_{1m}\bigr)&D_1^Ru_1&=&0,\\
            D_2^L\bigl(A_{20}& +&\lambda_1 A_{21}
            &+&\ldots&+& \lambda_m A_{2m}\bigr)&D_2^Ru_2&=&0,\\
            \vdots &&\vdots&&&&\vdots&\vdots&&\vdots\\
            D_m^L\bigl(A_{m0}& +&\lambda_1 A_{m1}
            &+&\ldots&+& \lambda_m A_{mm}\bigr)&D_m^Ru_m&=&0
        \end{array}         
    \end{equation}
    is a definite MEP as discussed in \Cref{sec:definiteMEPs}. The following considerations can also be applied to case of positive definite matrices $D_1^L, \ldots, D_m^L $ and $D_1^R, \ldots, D_m^R$ under the additional constraint that $D_k^L D_k^R = D_k^R D_k^L $ for $k = 1,\ldots, m$. 
    Let us assume that the eigenvalues of $ B_k(\lambda) = A_{k0}+\lambda_1 A_{k1}+\ldots \lambda_m A_{km}$ are real, and that $\varepsilon'_{k,i_k}(\lambda)$ is the $i_k$-th largest eigenvalue. If the eigenvalue is simple, its derivative can be calculated by 
    \begin{equation*}
        w_k^\He v^{}_k  \varepsilon'_{k,i_k}(\lambda)=\begin{pmatrix}
            w_k^\He A^{}_{k1} v^{}_k&
            \!\!\!\ldots\!\!\!
            &
            w_k^\He A^{}_{km} v^{}_k
        \end{pmatrix}
    \end{equation*}
    where $v_k$ and $w_k$  are corresponding right and left eigenvectors of $B_k(\lambda)$, and the analog of \Cref{alg:Newton} is given in \Cref{alg:Newtonnonhermitian}. 
    We treat an MEP with non-Hermitian matrices as definite if the transformed MEP~\eqref{eq: transformed MEP} is a definite MEP, as its properties are essentially the same.
    Indeed, \Cref{alg:Newtonnonhermitian} applied to such an MEP has the same iterates as \Cref{alg:Newton} applied to a similar MEP with Hermitian matrices.
    \begin{algorithm}[t]
        \caption{Semismooth Newton method for right definite MEPs with non-Hermitian matrices}\label{alg:Newtonnonhermitian}
        \begin{algorithmic}
        \Require A right definite MEP of the form~\eqref{eq:MEP} and a multiindex $\mathbf i$
        \Ensure Approximation of an eigenvalue $\lambda$ of multiindex $\mathbf i$ and corresponding right and left eigenvalues $v_1,\ldots, v_m$ and $w_1,\ldots, w_m$
        \State  initial guess of right eigenvector $v^{(0)}_1,\ldots,v^{(0)}_m$ and left eigenvector $w^{(0)}_1,\ldots,w^{(0)}_m$
        \For {$j=1,2,\dots$}
            \State
            solve $
            \tilde W(w^{(j-1)}_1,\ldots,w_m^{(j-1)}, v_1^{(j-1)},\ldots,v_m^{(j-1)})\begin{pmatrix}
                1\\
                \lambda^{(j)}
                \end{pmatrix}=0$
            \For {$k=1,\dots,m$}
                \State compute the $i_k$-th largest eigenvalue of $\sum_{\ell=0}^m \lambda^{(j-1)}_\ell A^{}_{k\ell}$ 
                \State compute a corresponding left eigenvector $w_k^{(j)}$ and right \State eigenvector $v^{(j)}_k\in\mathcal U_k$
                normalized such that $(w^{(j)}_k)^\He v^{(j)}_k= 1$
            \EndFor

        \EndFor
        \State
        \Return $\lambda^{(j)}$, $v^{(j)}_1,\ldots, v^{(j)}_m$ and $w^{(j)}_1,\ldots, w^{(j)}_m$
        \end{algorithmic}
        \end{algorithm}

        \begin{lemma}
            Let~\eqref{eq: transformed MEP} be a right definite MEP and let $\lambda\in \R^m$. Then the eigenvalues of 
            \[
            B_k = A_{k0}+\lambda_1 A_{k1}+\ldots \lambda_m A_{km}
            \]
            coincide with the eigenvalues of the Hermitian matrix 
            \[
            \tilde B_k = (D_k^R)^{-\frac12}(D_k^L)^{\frac12}B_k(D_k^L)^{-\frac12}(D_k^R)^{\frac12}.\] Furthermore, let $u_k$ be an eigenvector of $\tilde B_k$. Then $v_k = (D_k^L)^{-\frac12}(D_k^R)^{\frac12} u_k$
            is an eigenvector of $B_k$ and
            $w_k = (D_k^R)^{-\frac12}(D_k^L)^{\frac12} u_k$  is an eigenvector of $B_k^\He$, and $u_k^\He \tilde B_k u_k^{} = w_k^\He B_k v_k^{}$ hold.
        \end{lemma} 
            \begin{proof}
    The first claim follows as $B_k$ and $\tilde B_k$ are similar matrices. 
    The matrix $\tilde B_k$ is Hermitian as $(D_k^L)B_k(D_k^R)$ is Hermitian and 
    \[
     \tilde B_k = (D_k^R)^{-\frac12}(D_k^L)^{-\frac12} (D_k^L)B_k(D_k^R) (D_k^R)^{-\frac12}(D_k^L)^{-\frac12}
    \]
    using the fact that diagonal matrices commute.
    The second claim follows by the calculation
    \[
    B_k v_k = B_k (D_k^L)^{-\frac12}(D_k^R)^{\frac12} u_k
    = (D_k^L)^{-\frac12}(D_k^R)^{\frac12}
    \tilde B_k  u_k =
    (D_k^L)^{-\frac12}(D_k^R)^{\frac12} \varepsilon_k u_k = \varepsilon_k v_k.
    \]
    The claim for $w_k$ follows in analogy and  $u_k^\He \tilde B_k u_k^{} = w_k^\He B_k v_k^{}$ is a straightforward calculation. 
    \end{proof}

    This scenario is approximately the case when discretizing a Sturm-Liouville problem in a nonsymmetric way, or when it is easier to discretize a selfadjoint Sturm-Liouville problem after a change of variables. Consider for example the selfadjoint Sturm-Liouville problem
    \[
        p(x) u''(x) +p'(x) u'(x) +q(x) u(x) = \lambda w(x)u(x) \quad \text{for $x\in(a,b)$} 
    \]
    with homogeneous Dirichlet or Neumann boundary conditions. 
    Let $\{\varphi_i\}_{i = 1,\ldots, N}$ be a discrete basis, such that $\varphi_i(x_j) = \delta_{ij}$ for nodes $\{x_i\}_{i= 1,\ldots, N}$. Then the matrices
    \[
        \left(\int_a^b p(x) \varphi'_i(x) \varphi'_j(x) dx\right)_{ij}, \quad
        \left(\int_a^b q(x) \varphi_i(x) \varphi_j(x) dx\right)_{ij}, \quad\text{ and }\quad
        \left(\int_a^b w(x) \varphi_i(x) \varphi_j(x) dx\right)_{ij}
    \]
    are symmetric but require the evaluation of $N(N+1)/2$ integrals each. One can instead use the nonsymmetric matrices
    \begin{equation}\label{eq: collocation matrices}
        \left(p(x_j) \varphi''_i(x_j) + p'(x_j) \varphi'_i(x_j) \right)_{ij}, \quad
        \left(q(x_j) \varphi_i(x_j)\right)_{ij}, \quad\text{ and }\quad
        \left(w(x_j) \varphi_i(x_j)\right)_{ij},
    \end{equation}
    which are products of a diagonal matrix and differentiation matrices and are usually cheap to evaluate. The first ones can however be approximated by the later ones after multiplying with a diagonal weight matrix that approximates the integrals by a quadrature.

    One simple but important example where the matrices in~\eqref{eq: collocation matrices} satisfy~\eqref{eq: transformed MEP} even exactly is when the function $p$ is constant, the basis 
    $\{\varphi_i\}_{i= 1,\ldots, N}$ is a nodal basis of polynomials up to degree $N+1$ satisfying Dirichlet boundary conditions 
    $\varphi_i(a)= 0 =\varphi_i(b)$, and the nodes $\{x_i\}_{i= 1,\ldots, N}$ are the interior nodes of the Gauss-Lobatto quadrature rule. Then the matrices become symmetric after multiplying by a diagonal matrix containing the corresponding weights.

    The other situation arises, when the selfadjointness of the Sturm-Liouville problem is lost by a change of variables. Consider for example the problem of above with $\tilde u = u\circ\phi$, $\tilde p = p\circ \phi$, $\tilde q = q\circ \phi$, and $\tilde w = w\circ \phi$. The previous Stum-Liouville problem transforms to
    \[
    \tilde p \phi' \tilde u'' +(\tilde p' \phi' -\tilde p \phi'') \tilde u' + \tilde q \phi'^3 \tilde u=  \lambda  \tilde w \phi'^3\tilde u,
    \]
    which gets selfadjoint when dividing by $\phi'^2$. 

    A very similar strategy can be applied to a non-symmetric finite difference discretization. For example, the finite difference eigenvalue problem
    \[
        a_k u_{k+1} - b_k u_k +\tilde a_{k-1}u_{k-1} = \lambda u_k 
    \]
    with positive coefficient $a_k$ and $\tilde a_k$ can be made symmetric by left multiplication by a diagonal matrix with positive entries $d_k$ satisfying the relation $d_k a_k = d_{k+1}\tilde a_k$. Hence, a nonsymmetric multiparameter finite difference problem can be transformed into a symmetric one in a way leading to~\eqref{eq: transformed MEP}. 

    In \Cref{sec: numerical experiment Ellipsoidal wave equation}, we will apply \Cref{alg:Newtonnonhermitian} to a discretization of the Ellipsoidal wave equation, where both the equation is not in a selfadjoint form and the discretization is of the form~\eqref{eq: collocation matrices}. In this case, it is unlikely that~\eqref{eq: transformed MEP} is applicable exactly. However, since the underlying continuous problem~\eqref{eq:ellipsoiderMEP} is both left and right definite, it is not unlikely that \Cref{alg:Newtonnonhermitian} finds approximate solutions to the original problem.

    \section{Numerical experiments}\label{sec:numericalMEP}
    We compare the performance of the proposed Newton-type method to various method in the \textsc{Matlab} toolbox \textsc{MultiParEig}\cite{multipareig}.
    The numerical experiments were performed on a M1 Pro processor using a single kernel. 

    It should be noted that we have not implemented our methods with the highest efficiency. For example, all methods can be parallelized in multiple ways. The computation of an eigenvalue of a given multiindex is independent of the computation of an eigenvalue with a different multiindex. 
    These can therefore be computed in parallel.
    If only one eigenvalue is computed, the computation of the $m$ different eigenvalues in the second loop of the algorithms can also be done in parallel. 

    In all experiments, we measure the accuracy of eigenvalues and eigenvectors by the normalized residual error
    \[
        \max_{k = 1,\ldots, m} \frac{1}{\|u_k\|} \left\|\sum_{\ell = 0}^m \lambda_\ell A_{kl} u_k\right\|.   
    \]
    If we compute more than one eigenvalue, we always depict the maximum of the errors in the results.

    \subsection{Ellipsoidal wave equation}\label{sec: numerical experiment Ellipsoidal wave equation}

    In our initial experiment, we opted for a specific discretization of the Helmholtz equation on the ellipsoid with semi-axes $x_0, y_0$, and $z_0$, subject to Dirichlet boundary conditions, as expressed in~\eqref{eq:ellipsoiderMEP}. Defining $a=\sqrt{z_0^2-x_0^2}$, $b=\sqrt{z_0^2-y_0^2}$, $c=a^2/b^2$, and $\omega^2=4\eta/4$, we separated the problem into equations for $u_k(t_k)$, where $k = 1,2,3$. One form of separation leads to
    \begin{equation}
        t_k (t_k - 1)(t_k - c) u''_k (t_k) + \frac12 (3t^2_k - 2 (1 + c)t_k + c) u'_k
        (t_k) + (\lambda + \mu t_k + \eta t^2_k) u_k(t_k) = 0,        
    \end{equation}
    for $k = 1,2,3$ where $t_1 \in (c, z_0^2/b)$, $t_2 \in (1,c)$ and $t_3 \in (0,1)$. Here, $\lambda$ and $\mu$ are separation parameters. These equations have singularities at $0,1$ and $c$. Assuming that the equations are also satisfied at the boundary values $0, 1, c$ and the Dirichlet boundary condition $u_1(z_0^2/b^2) = 0$ leads to one of eight possible configurations of the original problem; see~\cite{Arscott83}. For our numerical experiment, we chose $x_0= 1, y_0 = 1.5$ and $z_0=2$ in accordance to the numerical experiments in~\cite{Plestenjak2015,HKP2004}. 
    
    We discretize using a polynomial basis of degree $N-1$ on $N$~Chebyshev nodes
    \[
    t^{(j)} = \frac{(b-a)\cos\left(\frac{j-1}{N-1}\right)}{2} + \frac{a+b}{2}
    \]
    for $j  = 1,\ldots, N$ and an interval of the form $(a,b)$. The discretized problem reads
    \[
        t^{(j)}_k (t^{(j)}_k - 1)(t^{(j)}_k - c) u''_k (t^{(j)}_k) + \frac12 (3(t^{(j)}_k)^2 - 2 (1 + c)t^{(j)}_k + c) u'_k
        (t^{(j)}_k) + (\lambda + \mu t^{(j)}_k + \eta (t^{(j)}_k)^2) u_k(t^{(j)}_k) = 0,
    \]
    where $k = 1,2,3$, the functions $u_k$ are in the space of polynomials of degree $N-1$, $j = 2, \ldots, N-1$ and the described boundary conditions hold.

    For our experiments, we chose $N = 200$. Note, that the resulting matrices are not symmetric. Hence, we use \Cref{alg:Newtonnonhermitian}. In this case, we did not use a globalization strategy.

    We aim to compute eigenvalues corresponding to the lowest frequencies~$\omega$, which in this case means finding eigenvalues with the smallest value in $\eta$. To this end, we use definiteness of the original selfadjoint problem~\eqref{eq:ellipsoiderMEP} with respect to $(0,0,0,1)$ and \Cref{prop:smallindices} in the following way:
    Let~$\mathcal E$ be a list of currently computed extreme eigenvalues, that is computed eigenvalues of multiindex~$\mathbf i$ such that we have not computed an eigenvalue of multiindex~$\mathbf j$ with~$\mathbf j>\mathbf i$. Choose the index $\mathbf i$ with the smallest value~$\eta$, and remove it from~$\mathcal E$. Compute the eigenvalues with index~$\mathbf i + e_i$ with~$e_i$ the coordinate unit vectors and add them to the list~$\mathcal E$.

    \begin{table}
        \caption{Eigenvalues $(\lambda, \mu, \eta)$ and corresponding multiindex $(i_1,i_2,i_3)$ computed by \Cref{alg:Newtonnonhermitian} of the ellipsoidal wave equation as described in \Cref{sec: numerical experiment Ellipsoidal wave equation} as well as the order in which they are computed by the described strategy.}\label{table: eigenvalues of ellipsoid}
        \begin{filecontents*}{Data/eigenTable.csv}
ord, LAMBDA, MU, ETA, ione,itwo,ithree
1,0.849892094168104,-3.7523178238368,2.40498181723942,1,1,1
4,7.22643743623491,-13.0312275640503,5.59866649175357,1,1,2
3,2.05458474627635,-13.4699482815672,7.46473319621638,1,2,1
7,17.3418202403812,-27.4541864476672,10.4953702044522,1,1,3
6,13.1546625314614,-28.4048067726918,12.3327312635814,1,2,2
2,2.18548886495183,-16.2260907978002,12.6541016391761,2,1,1
9,3.20211788479924,-29.3367106537258,15.9168668252089,1,3,1
12,31.6749201524686,-47.0816957920257,16.975396191381,1,1,4
5,13.7017476097789,-33.2720345076096,18.4634465167854,2,1,2
11,27.7125097293869,-48.3952724929107,18.9358871334632,1,2,3
8,3.36826005292507,-32.8386945621401,21.2369229619464,2,2,1
14,18.8905338021184,-50.1272957599894,22.5939845918323,1,3,2
20,50.6549704983939,-71.7276411328175,24.7988576040977,1,1,5
10,28.4114731432986,-54.9380124858796,25.8319374511914,2,1,3
19,45.6773910750939,-73.7527588894666,27.4539482230575,1,2,4
17,4.34444669242106,-51.2477963653999,27.5414058182087,1,4,1
13,19.6614676425232,-55.2270675953783,28.5469798941982,2,2,2
23,38.1444919962172,-75.7978811637857,30.8625050446722,1,3,3
15,3.54080209312227,-37.4386513989461,31.6355233533939,3,1,1
16,4.50075673423701,-55.2220117829999,33.0521513506502,2,3,1
\end{filecontents*}
\csvreader[
            head to column names,
            before reading = \begin{center}\sisetup{table-number-alignment=center},
            tabular =  ccc|ccc| c,
            table head = \toprule  $\lambda$ & $\mu$ & $\eta$& $i_1$ & $i_2$ & $i_3$ & order of computation\\\midrule,
            after reading = \end{center},
            ]{Data/eigenTable.csv}{}{%
             \tablenum[ round-precision=3,
            round-mode=places]{\LAMBDA} & \tablenum[ round-precision=3,
            round-mode=places]{\MU } &\tablenum[ round-precision=3,
            round-mode=places]{\ETA}&\tablenum{\ione} & \tablenum{\itwo} & \tablenum{\ithree} &\tablenum{\ord}
            }
    \end{table}
    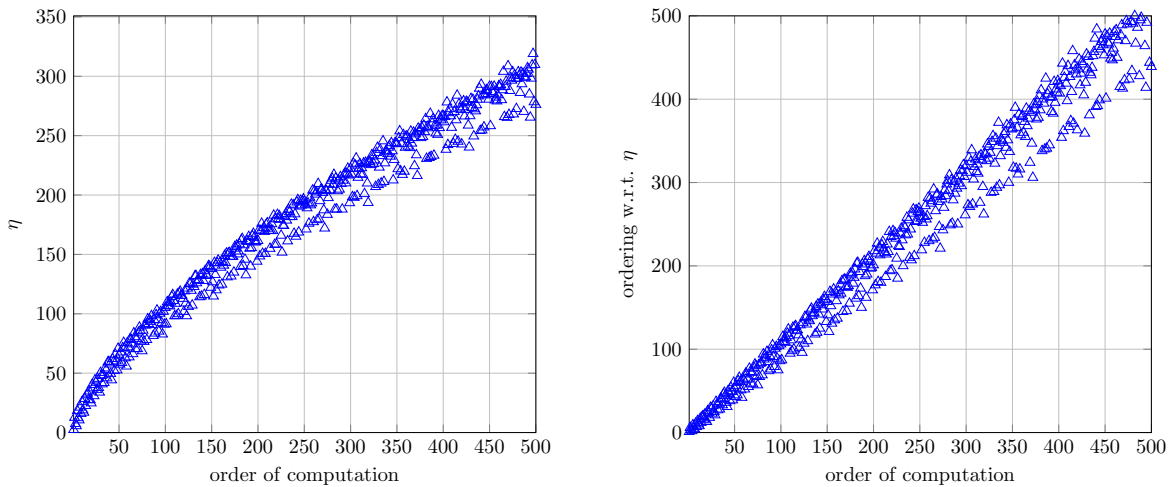
\begin{figure}
        \centering
\begin{subfigure}
  {0.496\textwidth}
   \begin{tikzpicture}[scale=0.67]
      \begin{axis}[
        width=1.3\textwidth,
        height=1.2\textwidth,
          grid=major,xlabel={order of computation},ylabel={ $\eta$},xmin=1, xmax=500,ymin = 0,
          legend style={at={(0.97,0.5)},anchor=east}] 
          \addplot[only marks,solid,color=blue,mark=triangle, mark size = 3pt ] %
          table[x=order,y=eta,col sep=comma]{Data/orderExperimentEllipsoid.csv};
          
      \end{axis}
  \end{tikzpicture}   
  \end{subfigure}
  \begin{subfigure}
    {0.496\textwidth}
     \begin{tikzpicture}[scale=0.67]
        \begin{axis}[
            width=1.3\textwidth,
            height=1.2\textwidth,
            grid=major,xlabel={order of computation},ylabel={ ordering w.r.t. $\eta$},xmin=1, xmax=500,ymin = 0,ymax = 500,
            legend style={at={(0.97,0.5)},anchor=east}] 
            \addplot[only marks,solid,color=blue,mark=triangle, mark size = 3pt ] %
            table[x=order,y=sort,col sep=comma]{Data/orderExperimentEllipsoid.csv};  
        \end{axis}
    \end{tikzpicture}   
    \end{subfigure}
        \caption{Experiment for the ellipsoidal wave equation as described in \Cref{sec: numerical experiment Ellipsoidal wave equation}. This figure shows the order of computation. }\label{fig:expEllipsOrder}
    \end{figure}
    \begin{figure}
        \centering
        \begin{subfigure}
        {0.496\textwidth}
        \begin{tikzpicture}[scale=0.67]
        \begin{axis}[
        width=1.3\textwidth,
          height=1.2\textwidth,
        grid=major,xlabel={number of computed eigenvalues},ylabel={ time in seconds},xmin=10, xmax=310,ymin = 0,
        legend style={at={(0.97,0.5)},anchor=east}] 
        \addplot[only marks,solid,color=blue,mark=triangle, mark size = 3pt ] %
        table[x=N,y=time,col sep=comma]{Data/errorTimeTableNewtonEllipsoidal.csv};
        \addplot[only marks,solid,color=red,mark=square, mark size = 2pt] %
        table[x=N,y=time1,col sep=comma]{Data/errorTimeTableJDEllipsoidal.csv};
        \addplot[only marks,solid,color=red,mark=square, mark size = 2pt] %
        table[x=N,y=time2,col sep=comma]{Data/errorTimeTableJDEllipsoidal.csv};
        \addplot[only marks,solid,color=red,mark=square, mark size = 2pt] %
        table[x=N,y=time3,col sep=comma]{Data/errorTimeTableJDEllipsoidal.csv};
        
        \end{axis}
        \end{tikzpicture}   
        
        \end{subfigure}
        \begin{subfigure}{0.496\textwidth}
        
            \begin{tikzpicture}[scale=0.67]
                \begin{semilogyaxis}[
                width=1.3\textwidth,
                  height=1.2\textwidth,
                grid=major,xlabel={number of computed eigenvalues},ylabel={ error},xmin=10, xmax=310,
                legend style={at={(0.97,0.5)},anchor=east}] 
                \addplot[only marks,solid,color=blue,mark=triangle, mark size = 3pt] %
                table[x=N,y=error,col sep=comma]{Data/errorTimeTableNewtonEllipsoidal.csv};
                \addlegendentry{\Cref{alg:Newtonnonhermitian}}
                \addplot[only marks,solid,color=red,mark=square, mark size =2pt] %
                table[x=N,y=error,col sep=comma]{Data/errorTimeTableJDEllipsoidal.csv};
                \addlegendentry{\textsc{threepareigs\_jd}}

                \end{semilogyaxis}
                \end{tikzpicture}

        \end{subfigure}
        \caption{Experiment for the ellipsoidal wave equation as described in \Cref{sec: numerical experiment Ellipsoidal wave equation}. The left figure compares the computational complexity with  \textsc{threepareigs\_jd} from~\cite{multipareig} and the right figure the acquired accuracy of eigenvalues.}\label{fig:expEllips}
            
    \end{figure}
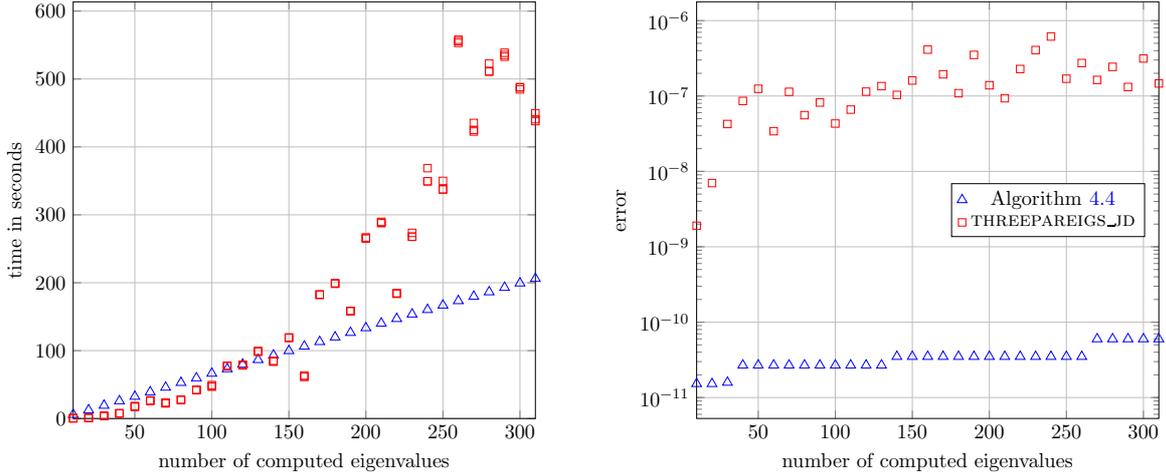

In \Cref{table: eigenvalues of ellipsoid}, we show the eigenvalues with smallest $\eta$ as well as their multiindices and the order of computation. At first, we note that the eigenvalues agree with the eigenvalues shown in~\cite[Table 6]{Plestenjak2015} for the configuration $\rho = \sigma = \tau =0$, which corresponds to our numerical experiment. The described strategy to find eigenvalues with small~$\eta$ works reasonably well. To compute the $20$  eigenvalues, we had to compute $23$ eigenvalues, for the smallest $13$ eigenvalues, we required to compute $20$ eigenvalues. The further behavior is seen in \Cref{fig:expEllipsOrder}. It is necessary to compute roughly $500$ eigenvalues to find the $400$ eigenvalues with smallest $\eta$ and this behavior seems to continue.

Furthermore, we compared our method to the subspace method \textsc{threepareigs\_jd} from the \textsc{Matlab} toolbox \textsc{MultiParEig}~\cite{multipareig}. To this end we computed the $k = 10, 20, \dots, 300, 310$ eigenvalues with small $\eta$ with both methods. We compared the runtime and the error in the eigenvalues. The runtime of our method is linear in the required number of eigenvalues, while the subspace method \textsc{threepareigs\_jd} is more time-consuming when requiring many eigenvalues. This is to be expected, as the corresponding eigenvectors to eigenvalues with small frequencies~$\omega$ do not oscillate a lot and can therefore be approximated by a smaller common subspace, while eigenvectors to slightly larger eigenvectors have successively higher oscillation but in different modes. This leads to smaller computational costs when a moderate number of eigenvalues needs to be computed.

The maximal error of the eigenvalues is considerably smaller with our method, and does not deteriorate for eigenvalues with larger frequencies. 
These findings are presented in \Cref{fig:expEllips}.

We want to remark that we have not incorporated good initialization into this experiment. Small examples indicate that the number of iterations per eigenvalue can be halved when a similar strategy as in~\cite{GHP_Math_2012} is applied.

\subsection{Randomly generated examples}\label{sec: randomly generated examples}

To further test the performance of \Cref{alg:Newton} or \Cref{alg:NewtonGlobalized} respectively, we generated definite multiparameter problems randomly. To this end, we generated the matrices $A_{10}, A_{20}, \ldots, A_{m0}$ as symmetric matrices with normally distributed entries.

For the other matrices, we took inspiration from~\eqref{eq:ellipsoiderMEP}. We generate diagonal matrices $D_1, D_2, \ldots, D_m$, where $D_k$ has uniformly distributed entries in the interval $[k-1, k]$. 
Now observe that the MEP with $A_{k\ell} = D_k ^{\ell-1}$ is definite with probability $1$ as a Vandermonde matrix is invertible.
The condition number of Vandermonde matrices increases rapidly with the number of nodes, especially when the modulus of the nodes are larger than one. Similarly, the condition number of the matrices appearing in the Newton iteration increases which possibly leads to slower convergence and less accuracy of the solution. To mitigate this problem slightly, we use Laguerre polynomials rather than monomials. The condition number of matrices with entries determined by Laguerre polynomials at different nodes still increases quickly, but not as rapidly as it is the case for Vandermonde matrices.
This results in better conditioned linear systems arising when solving multiparameter problems by Newton's method, but they do still deteriorate rather quickly with larger values of~$m$.
We thus construct the other matrices as the diagonal matrices  
\begin{equation}\label{eq: Laguerre random matrices}
    A_{k\ell} = L_{\ell- 1}(D_k)    
\end{equation}
where $L_\ell$ is the $\ell$-th Laguerre polynomial. This results in a definite multiparameter problem with probability $1$ as well, since it is just a linear transformation in the eigenvalues of the MEP with monomial coefficient matrices.

These multiparameter eigenvalue problems are increasingly bad conditioned for large $m$. To also see the case of large $m$ and well conditioned problems, we generate the matrices $A_{k\ell}$ for $\ell =1,\ldots, m$ as 
\begin{equation}\label{eq: uniformly well conditioned matrices}
    A_{k\ell} = Q_{k\ell}   D_{k\ell} Q_{k\ell}^\T + \delta_{k\ell} I_n,
\end{equation}
where $D_{k\ell}$ are diagonal matrices with entries uniformly distributed in the interval $[-\frac{1}{2m},\frac{1}{2m}]$, $Q_{k\ell}$ are randomly generated orthogonal matrices, $I_n$ is an $n$ dimensional identity matrix 
and $\delta_{k\ell}$ is the Kronecker delta.
This leads to uniformly well conditioned linear systems.

\subsubsection{Randomly generated three-parameter eigenvalue problems}
\begin{figure}
    \centering
    \begin{subfigure}
    {0.496\textwidth}
     \begin{tikzpicture}[scale=0.67]
    \begin{loglogaxis}[
        xtick={2,4,8,16},
    xticklabels={$2^1$,$2^2$,$2^3$,$2^4$},
    width=1.3\textwidth,
      height=1.2\textwidth,
    grid=major,xlabel={size of matrices $n$},ylabel={ time in seconds},xmin=2, xmax=24,ymin = 1e-4,ymax = 1600,
    legend style={at={(0.97,0.03)},anchor=south east}] 
    \addplot[samples=3,domain=2:24,blue, dashed] {0.4e-5*x^5};
    \addlegendentry{Rate $n^5$ }
    \addplot[samples=3,domain=2:16,red, dashed] {1e-8*x^9};
    \addlegendentry{Rate $n^9$}
    \foreach \i in {2,...,16}
    {
        \addplot[only marks,solid,color=blue,mark=triangle, mark size = 3pt ] %
        table[x=n,y=timeN,col sep=comma]{Data/RandomTestThree/errorTimeTableRandom\i.csv};
        \addplot[only marks,solid,color=red,mark=square, mark size = 2pt] %
        table[x=n,y=timeC,col sep=comma]{Data/RandomTestThree/errorTimeTableRandom\i.csv};
    }
    \foreach \i in {17,...,24}
    {
        \addplot[only marks,solid,color=blue,mark=triangle, mark size = 3pt ] %
        table[x=n,y=timeN,col sep=comma]{Data/RandomTestThree/errorTimeTableRandom\i.csv};
    }
    \end{loglogaxis}
    \end{tikzpicture}   
    \end{subfigure}
    \begin{subfigure}{0.496\textwidth}
        \begin{tikzpicture}[scale=0.67]
            \begin{semilogyaxis}[
            width=1.3\textwidth,
              height=1.2\textwidth,
            grid=major,xlabel={size of matrices $n$},ylabel={ error},xmin=2, xmax=24, 
            legend style={at={(0.97,0.5)},anchor=east}] 
            \addplot[only marks,solid,color=blue,mark=triangle, mark size = 3pt ] %
            table[x=n,y=errN,col sep=comma]{Data/RandomTestThree/errorTimeTableRandom2.csv};
            \addlegendentry{\Cref{alg:NewtonGlobalized}}
            \addplot[only marks,solid,color=red,mark=square, mark size = 2pt] %
            table[x=n,y=errC,col sep=comma]{Data/RandomTestThree/errorTimeTableRandom2.csv};
            \addlegendentry{{\textsc{threepareig}}}
            \foreach \i in {3,...,16}{
                \addplot[only marks,solid,color=blue,mark=triangle, mark size = 3pt ] %
            table[x=n,y=errN,col sep=comma]{Data/RandomTestThree/errorTimeTableRandom\i.csv};
            \addplot[only marks,solid,color=red,mark=square, mark size = 2pt] %
            table[x=n,y=errC,col sep=comma]{Data/RandomTestThree/errorTimeTableRandom\i.csv};
            }
            \foreach \i in {17,...,24}{
                \addplot[only marks,solid,color=blue,mark=triangle, mark size = 3pt ] %
            table[x=n,y=errN,col sep=comma]{Data/RandomTestThree/errorTimeTableRandom\i.csv};
            }
            \end{semilogyaxis}
            \end{tikzpicture} 
    \end{subfigure}
    \centering
    \begin{subfigure}
    {0.496\textwidth}
     \begin{tikzpicture}[scale=0.67]
    \begin{loglogaxis}[
        xtick={2,4,8,16},
    xticklabels={$2^1$,$2^2$,$2^3$,$2^4$},
    width=1.3\textwidth,
      height=1.2\textwidth,
    grid=major,xlabel={size of matrices $n$},ylabel={ time in seconds},xmin=2, xmax=24,ymin = 1e-4,ymax = 1600,
    legend style={at={(0.97,0.03)},anchor=south east}] 
    \addplot[samples=3,domain=2:24,blue, dashed] {0.2e-4*x^5};
    \addlegendentry{Rate $n^5$ }
    \addplot[samples=3,domain=2:16,red, dashed] {1e-8*x^9};
    \addlegendentry{Rate $n^9$}
    \foreach \i in {2,...,16}
    {
        \addplot[only marks,solid,color=blue,mark=triangle, mark size = 3pt ] %
        table[x=n,y=timeN,col sep=comma]{Data/RandomTestThree/errorTimeTableRandomWC\i.csv};
        \addplot[only marks,solid,color=red,mark=square, mark size = 2pt] %
        table[x=n,y=timeC,col sep=comma]{Data/RandomTestThree/errorTimeTableRandomWC\i.csv};
    }

    \foreach \i in {17,...,24}
    {
        \addplot[only marks,solid,color=blue,mark=triangle, mark size = 3pt ] %
        table[x=n,y=timeN,col sep=comma]{Data/RandomTestThree/errorTimeTableRandomWC\i.csv};
    }
    \end{loglogaxis}
    \end{tikzpicture}   
    \end{subfigure}
    \begin{subfigure}{0.496\textwidth}
        \begin{tikzpicture}[scale=0.67]
            \begin{semilogyaxis}[
            width=1.3\textwidth,
              height=1.2\textwidth,
            grid=major,xlabel={size of matrices $n$},ylabel={ error},xmin=2, xmax=24, 
            legend style={at={(0.97,0.5)},anchor=east}] 
            \addplot[only marks,solid,color=blue,mark=triangle, mark size = 3pt ] %
            table[x=n,y=errN,col sep=comma]{Data/RandomTestThree/errorTimeTableRandomWC2.csv};
            \addlegendentry{\Cref{alg:Newton}}
            \addplot[only marks,solid,color=red,mark=square, mark size = 2pt] %
            table[x=n,y=errC,col sep=comma]{Data/RandomTestThree/errorTimeTableRandomWC2.csv};
            \addlegendentry{{\textsc{threepareig}}}
            \foreach \i in {3,...,16}{
                \addplot[only marks,solid,color=blue,mark=triangle, mark size = 3pt ] %
            table[x=n,y=errN,col sep=comma]{Data/RandomTestThree/errorTimeTableRandomWC\i.csv};
            \addplot[only marks,solid,color=red,mark=square, mark size = 2pt] %
            table[x=n,y=errC,col sep=comma]{Data/RandomTestThree/errorTimeTableRandomWC\i.csv};
            }
            \foreach \i in {17,...,24}{
                \addplot[only marks,solid,color=blue,mark=triangle, mark size = 3pt ] %
            table[x=n,y=errN,col sep=comma]{Data/RandomTestThree/errorTimeTableRandomWC\i.csv};
            }
            \end{semilogyaxis}
            \end{tikzpicture}  
    \end{subfigure}
    \caption{Comparison of the performance of  \Cref{alg:Newton} or \Cref{alg:NewtonGlobalized} respectively and \textsc{threepareig} from~\cite{multipareig}. The multiparameter eigenvalue problems are generated randomly as described in \Cref{sec: randomly generated examples} with $n\times n$ matrices and varying size $n$. Here, all eigenvalues are computed.
    The first two graphs showcase the case of increasingly bad conditioned problems generated with Laguerre polynomials~\eqref{eq: Laguerre random matrices}, the second two graphs show the uniformly well conditioned problems~\eqref{eq: uniformly well conditioned matrices}.
    }\label{fig:expRandom3}
\end{figure}
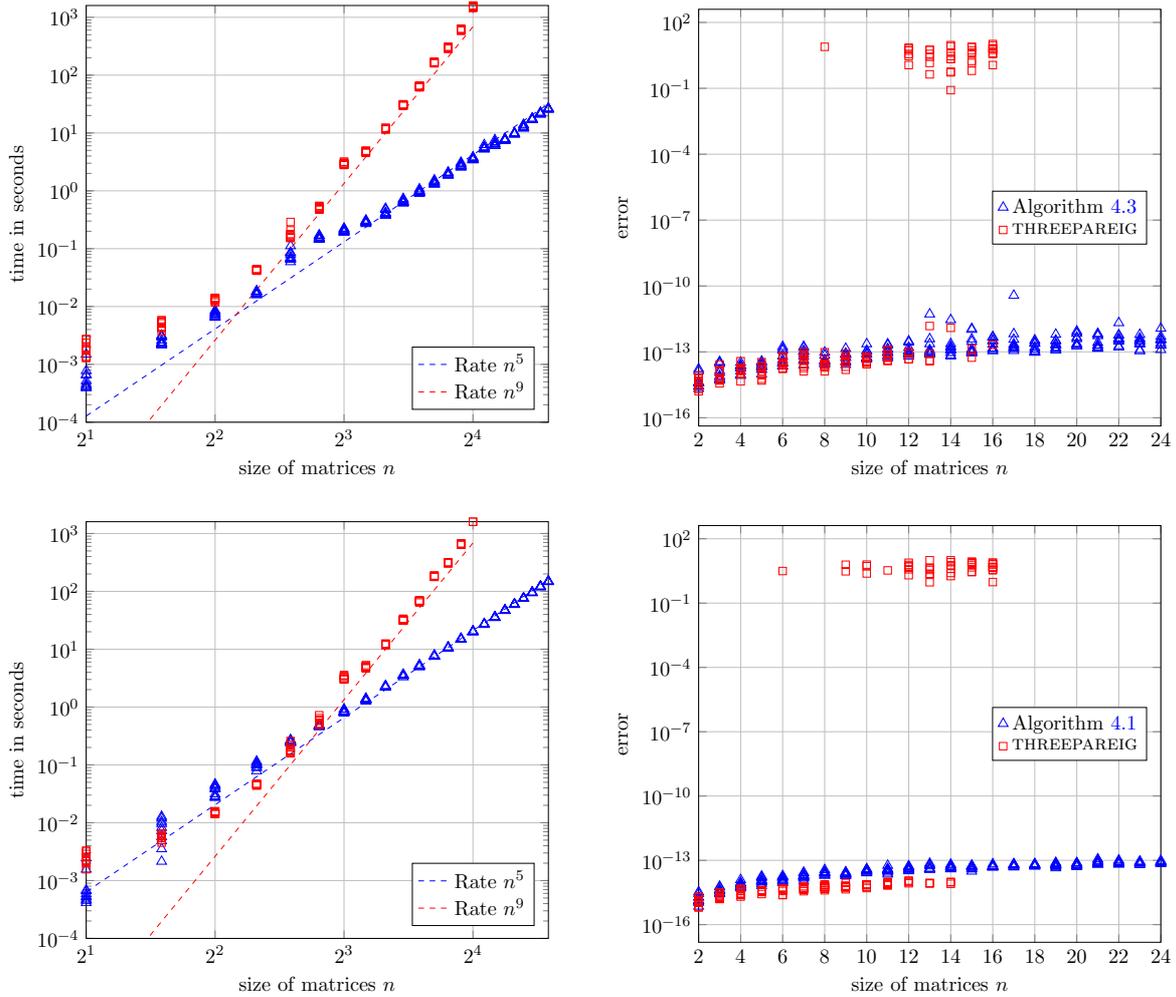 
As a next experiment, we generated $10$ different three-parameter eigenvalue problems with $n\times n$~matrices for each $n = 2,3, \ldots, 24$ and for both~\eqref{eq: Laguerre random matrices} and~\eqref{eq: uniformly well conditioned matrices} as described in the previous section. We noticed that in the case of~\eqref{eq: Laguerre random matrices}, we required a globalization strategy to reliably find all eigenvalues. Hence, we used \Cref{alg:Newton} for~\eqref{eq: uniformly well conditioned matrices} and \Cref{alg:NewtonGlobalized} for~\eqref{eq: Laguerre random matrices}
with $\tau = \frac{1}{2}$.
We used this method to compute all eigenvalues of the problem and compared the performance to \textsc{threepareig} from~\cite{multipareig}. The results are shown in \Cref{fig:expRandom3}.

The computation time of \Cref{alg:Newton} and \Cref{alg:NewtonGlobalized} are of order~$n^5$. We expect an order of $n^6$ as the method requires solving eigenvalue problems of~$n\times n$ matrices for~$n^3$ eigenvalue problems. However, for small $n$ the cost for solving an eigenvalue problem is rather of order~$n^2$ than~$n^3$. The computation time of \textsc{threepareig} is of order~$n^9$ as it requires solving generalized eigenvalue problems of~$n^3\times n^3$ problems. 

We do not use \textsc{threepareig} to solve problems with $n>16$ as the computation time is prohibitive. The resulting error in eigenvalues is slightly smaller but of the same order of magnitude for the ones computed with \Cref{alg:NewtonGlobalized}, when all eigenvalues are found. However, \textsc{threepareig} has problems finding all eigenvalues to satisfying precision while \Cref{alg:NewtonGlobalized} is able to find all eigenvalues. The same is true for \Cref{alg:Newton} in the case of randomly generated problems of the form~\eqref{eq: uniformly well conditioned matrices}, where the Newton method finds all eigenvalues but \textsc{threepareig} cannot. We want to note that in \cite[\examp 6.1]{HKP2024} a new version of  \textsc{threepareig}~\cite{multipareignew} was able to find all eigenvalues for the three parameter eigenvalue problems of the form described in~\eqref{eq: uniformly well conditioned matrices}.

\begin{figure}

    \centering
    \begin{subfigure}
    {0.496\textwidth}
     \begin{tikzpicture}[scale=0.67]
    \begin{semilogyaxis}[
    width=1.3\textwidth,
      height=1.2\textwidth,
    grid=major,xlabel={number of parameters $m$},ylabel={ time in seconds},xmin=2, xmax=15,ymin = 3*1e-5,ymax = 400,
    legend style={at={(0.97,0.03)},anchor=south east}] 
    
        \addplot[only marks,mark = triangle,color = blue, mark size = 3pt ] %
        table[x=m,y=time,col sep=comma]{Data/RandomTestVaryingM/testRandom2.csv};
        \addlegendentry{$n = 2$}
        \addplot[sharp plot, color = blue] %
        table[x=m,y expr = 1.5*(2^2*x+0.3*x^3)*2^x *1e-6,col sep=comma, forget plot]{Data/RandomTestVaryingM/testRandom2.csv};
    
        \addplot[only marks,mark = square,color = red, mark size = 3pt ] %
        table[x=m,y=time,col sep=comma]{Data/RandomTestVaryingM/testRandom3.csv};
        \addlegendentry{$n = 3$}
        \addplot[sharp plot, color = red] %
        table[x=m,y expr = 1.5*(3^2*x+0.3*x^3)*3^x *1e-6,col sep=comma, forget plot]{Data/RandomTestVaryingM/testRandom2.csv};
    
        \addplot[only marks,mark = diamond,color = olive, mark size = 3pt ] %
        table[x=m,y=time,col sep=comma]{Data/RandomTestVaryingM/testRandom4.csv};
        \addlegendentry{$n = 4$}
        \addplot[sharp plot, color = olive] %
        table[x=m,y expr = 1.5*(4^2*x+0.3*x^3)*4^x *1e-6,col sep=comma, forget plot]{Data/RandomTestVaryingM/testRandom2.csv};
    
        \addplot[only marks,mark = Mercedes star,color = purple, mark size = 3pt ] %
        table[x=m,y=time,col sep=comma]{Data/RandomTestVaryingM/testRandom5.csv};
        \addlegendentry{$n = 5$}
        \addplot[sharp plot, color = purple] %
        table[x=m,y expr = 1.5*(5^2*x+0.3*x^3)*5^x *1e-6,col sep=comma, forget plot]{Data/RandomTestVaryingM/testRandom2.csv};

        \addplot[only marks,mark = o,color = violet, mark size = 3pt ] %
        table[x=m,y=time,col sep=comma]{Data/RandomTestVaryingM/testRandom6.csv};
        \addlegendentry{$n = 6$}
        \addplot[sharp plot, color = violet] %
        table[x=m,y expr = 1.5*(6^2*x+0.3*x^3)*6^x *1e-6,col sep=comma, forget plot]{Data/RandomTestVaryingM/testRandom2.csv};
    \end{semilogyaxis}
    \end{tikzpicture}   
    
    \end{subfigure}
    \begin{subfigure}{0.496\textwidth}
    
        \begin{tikzpicture}[scale=0.67]
            
            \begin{semilogyaxis}[
            width=1.3\textwidth,
              height=1.2\textwidth,
            grid=major,xlabel={number of parameters $m$},ylabel={ error},xmin=2, xmax=15, 
            legend style={at={(0.97,0.5)},anchor=east}]

        \addplot[only marks,mark = triangle,color = blue, mark size = 3pt ] %
        table[x=m,y=error,col sep=comma]{Data/RandomTestVaryingM/testRandom2.csv};

        \addplot[only marks,mark = square,color = red, mark size = 3pt ] %
        table[x=m,y=error,col sep=comma]{Data/RandomTestVaryingM/testRandom3.csv};

        \addplot[only marks,mark = diamond,color = olive, mark size = 3pt ] %
        table[x=m,y=error,col sep=comma]{Data/RandomTestVaryingM/testRandom4.csv};

        \addplot[only marks,mark = Mercedes star,color = purple, mark size = 3pt ] %
        table[x=m,y=error,col sep=comma]{Data/RandomTestVaryingM/testRandom5.csv};

        \addplot[only marks,mark = o,color = violet, mark size = 3pt ] %
        table[x=m,y=error,col sep=comma]{Data/RandomTestVaryingM/testRandom6.csv};

            \end{semilogyaxis}
            \end{tikzpicture}

    \end{subfigure}
    \centering
\begin{subfigure}
{0.496\textwidth}
 \begin{tikzpicture}[scale=0.67]
\begin{semilogyaxis}[
width=1.3\textwidth,
  height=1.2\textwidth,
grid=major,xlabel={number of parameters $m$},ylabel={ time in seconds},xmin=2, xmax=15,ymin = 3*1e-5,ymax = 100,
legend style={at={(0.97,0.03)},anchor=south east}] 

    \addplot[only marks,mark = triangle,color = blue, mark size = 3pt ] %
    table[x=m,y=time,col sep=comma]{Data/RandomTestVaryingM/testRandomWC2.csv};
    \addlegendentry{$n = 2$}
    \addplot[sharp plot, color = blue] %
    table[x=m,y expr = (4*x+0.3*x^3)*2^x *1e-6,col sep=comma, forget plot]{Data/RandomTestVaryingM/testRandom2.csv};

    \addplot[only marks,mark = square,color = red, mark size = 3pt ] %
    table[x=m,y=time,col sep=comma]{Data/RandomTestVaryingM/testRandomWC3.csv};
    \addlegendentry{$n = 3$}
    \addplot[sharp plot, color = red] %
    table[x=m,y expr = (9*x+0.3*x^3)*3^x *1e-6,col sep=comma, forget plot]{Data/RandomTestVaryingM/testRandom2.csv};

    \addplot[only marks,mark = diamond,color = olive, mark size = 3pt ] %
    table[x=m,y=time,col sep=comma]{Data/RandomTestVaryingM/testRandomWC4.csv};
    \addlegendentry{$n = 4$}
    \addplot[sharp plot, color = olive] %
    table[x=m,y expr = (16*x+0.3*x^3)*4^x *1e-6,col sep=comma, forget plot]{Data/RandomTestVaryingM/testRandom2.csv};

    \addplot[only marks,mark = Mercedes star,color = purple, mark size = 3pt ] %
    table[x=m,y=time,col sep=comma]{Data/RandomTestVaryingM/testRandomWC5.csv};
    \addlegendentry{$n = 5$}
    \addplot[sharp plot, color = purple] %
    table[x=m,y expr = (25*x+0.3*x^3)*5^x *1e-6,col sep=comma, forget plot]{Data/RandomTestVaryingM/testRandom2.csv};

    \addplot[only marks,mark = o,color = violet, mark size = 3pt ] %
    table[x=m,y=time,col sep=comma]{Data/RandomTestVaryingM/testRandomWC6.csv};
    \addlegendentry{$n = 6$}
    \addplot[sharp plot, color = violet] %
    table[x=m,y expr = (36*x+0.3*x^3)*6^x *1e-6,col sep=comma, forget plot]{Data/RandomTestVaryingM/testRandom2.csv};
\end{semilogyaxis}
\end{tikzpicture}   

\end{subfigure}
\begin{subfigure}{0.496\textwidth}

    \begin{tikzpicture}[scale=0.67]
        
        \begin{semilogyaxis}[
        width=1.3\textwidth,
          height=1.2\textwidth,
        grid=major,xlabel={number of parameters $m$},ylabel={ error},xmin=2, xmax=15, 
        legend style={at={(0.97,0.5)},anchor=east}]

    \addplot[only marks,mark = triangle,color = blue, mark size = 3pt ] %
    table[x=m,y=error,col sep=comma]{Data/RandomTestVaryingM/testRandomWC2.csv};

    \addplot[only marks,mark = square,color = red, mark size = 3pt ] %
    table[x=m,y=error,col sep=comma]{Data/RandomTestVaryingM/testRandomWC3.csv};

    \addplot[only marks,mark = diamond,color = olive, mark size = 3pt ] %
    table[x=m,y=error,col sep=comma]{Data/RandomTestVaryingM/testRandomWC4.csv};

    \addplot[only marks,mark = Mercedes star,color = purple, mark size = 3pt ] %
    table[x=m,y=error,col sep=comma]{Data/RandomTestVaryingM/testRandomWC5.csv};

    \addplot[only marks,mark = o,color = violet, mark size = 3pt ] %
    table[x=m,y=error,col sep=comma]{Data/RandomTestVaryingM/testRandomWC6.csv};

        \end{semilogyaxis}
        \end{tikzpicture}

\end{subfigure}
\caption{Showcase of the performance of \Cref{alg:NewtonGlobalized} for $n\times n$ matrices and varying number of parameters $m$. The multiparameter eigenvalue problems were generated randomly as  described in \Cref{sec: randomly generated examples}. All eigenvalues were computed.  
The first two graphs show the case of increasingly bad conditioned problems generated with Laguerre polynomials~\eqref{eq: Laguerre random matrices}, the second two graphs showcase the uniformly well conditioned problems~\eqref{eq: uniformly well conditioned matrices}
As a comparison, we show the observed costs~$(mn^2+m^3)n^m$.}\label{fig:expRandomVaryingM}
    
\end{figure}

\subsubsection{Randomly generated multiparameter eigenvalue problems}
We generate $9$ multiparameter eigenvalue problems as described in \Cref{sec: randomly generated examples} for each combination of the following parameters:
\begin{center}
    \begin{tabular}{| c | c | c | c | c | c |}
        \hline
        $n  $ & $2$ &$3$ &$4$ &$5$ & $6$ \\
        $m  $ & $2,\dots, 15$ & $2,\dots, 10$ & $2,\dots, 8$ & $2,\dots, 7$ & $2,\dots, 6$\\
        \hline
    \end{tabular}        
\end{center}
We again aimed to compute all eigenvalues. The results are showcased in \Cref{fig:expRandomVaryingM}. In most cases, the method is able to find all eigenvalues, however the precision deteriorates with larger $m$. This is expected behavior, as the conditioning of the resulting linear systems also deteriorates. 

Here, for simplicity, we only used the globalized version. We note, that we implemented \Cref{alg:NewtonGlobalized} with the stopping criterion of reaching either an error of $10^{-11}$ or after reaching $40$ steps in the for-loop. In the case of randomly generated problems with~\eqref{eq: Laguerre random matrices} and for large $m$ the second stopping criterion is reached more often, leading at first to a deterioration of costs, and later to non convergence of some eigenvalues. This is not the case for the problems generated with~\eqref{eq: uniformly well conditioned matrices}. Here, the error does not deteriorate, and the costs behave as expected.

To test larger matrix sizes, we generate $9$ multiparameter eigenvalue problems for each $m = 2, \ldots, 15$ and~$n=25, 50, 100, 200, 400$, but this time we only compute one single eigenvalue with multiindex~$\mathbf 1 = (1,\ldots, 1)$. The results are shown in \Cref{fig:expRandomVaryingMSingle}. Again, the error deteriorates for larger~$m$ in the case of randomly generated problems  as expected which also leads to a visible deterioration of computation time. This starts noticeable for $m = 10$, when the error is larger than $10^{-11}$ more often than not, and thus the full $40$ steps of the for-loop are used.

For both experiments, the computational cost depends on the parameters $m$ and $n$ roughly as expected. For the values of $n$, the cost of solving an eigenvalue problem still scales like $n^2$ and $n^{2.5}$ respectively, rather than $n^3$. 
Therefore, we observe costs per eigenvalue $mn^2+m^3$ and $mn^{2.5}+m^3$ respectively, which is visible in both \Cref{fig:expRandomVaryingM} and \Cref{fig:expRandomVaryingMSingle}. The asymptotic costs for $n\to \infty$ should be  $mn^3+m^3$.

\begin{figure}
    \centering
\begin{subfigure}
{0.496\textwidth}
 \begin{tikzpicture}[scale=0.67]
\begin{loglogaxis}[
    xtick={2,4,8,16},
xticklabels={$2^1$,$2^2$,$2^3$,$2^4$},
width=1.3\textwidth,
  height=1.2\textwidth,
grid=major,xlabel={number of parameters $m$},ylabel={ time in seconds},xmin=2, xmax=15,ymin = 1e-5,ymax = 1600,
legend style={at={(0.97,0.03)},anchor=south east}] 

    \addplot[only marks,mark = triangle,color = blue, mark size = 3pt ] %
    table[x=m,y=time,col sep=comma]{Data/RandomTestVaryingMsingle/testRandomSingle25.csv};
    \addlegendentry{$n = 25$}
    \addplot[sharp plot, color = blue] %
    table[x=m,y expr = 7*(25^2*x+x^3) *1e-7,col sep=comma, forget plot]{Data/RandomTestVaryingM/testRandom2.csv};

    \addplot[only marks,mark = square,color = red, mark size = 3pt ] %
    table[x=m,y=time,col sep=comma]{Data/RandomTestVaryingMsingle/testRandomSingle50.csv};
    \addlegendentry{$n = 50$}
    \addplot[sharp plot, color = red] %
    table[x=m,y expr = 7*(50^2*x+x^3) *1e-7,col sep=comma, forget plot]{Data/RandomTestVaryingM/testRandom2.csv};

    \addplot[only marks,mark = diamond,color = olive, mark size = 3pt ] %
    table[x=m,y=time,col sep=comma]{Data/RandomTestVaryingMsingle/testRandomSingle100.csv};
    \addlegendentry{$n = 100$}
    \addplot[sharp plot, color = olive] %
    table[x=m,y expr = 7*(100^2*x+x^3) *1e-7,col sep=comma, forget plot]{Data/RandomTestVaryingM/testRandom2.csv};

    \addplot[only marks,mark = Mercedes star,color = purple, mark size = 3pt ] %
    table[x=m,y=time,col sep=comma]{Data/RandomTestVaryingMsingle/testRandomSingle200.csv};
    \addlegendentry{$n = 200$}
    \addplot[sharp plot, color = purple] %
    table[x=m,y expr = 7*(200^2*x+x^3) *1e-7,col sep=comma, forget plot]{Data/RandomTestVaryingM/testRandom2.csv};

    \addplot[only marks,mark = o,color = violet, mark size = 3pt ] %
    table[x=m,y=time,col sep=comma]{Data/RandomTestVaryingMsingle/testRandomSingle400.csv};
    \addlegendentry{$n = 400$}
    \addplot[sharp plot, color = violet] %
    table[x=m,y expr = 7*(400^2*x+x^3) *1e-7,col sep=comma, forget plot]{Data/RandomTestVaryingM/testRandom2.csv};

\end{loglogaxis}
\end{tikzpicture}   

\end{subfigure}
\begin{subfigure}{0.496\textwidth}

    \begin{tikzpicture}[scale=0.67]
        
        \begin{semilogyaxis}[
        width=1.3\textwidth,
          height=1.2\textwidth,
        grid=major,xlabel={number of parameters $m$},ylabel={ error},xmin=2, xmax=15, 
        legend style={at={(0.97,0.03)},anchor=south east}]

        \addplot[only marks,mark = triangle,color = blue, mark size = 3pt ] %
        table[x=m,y=error,col sep=comma]{Data/RandomTestVaryingMsingle/testRandomSingle25.csv};
        \addlegendentry{$n = 25$}

        \addplot[only marks,mark = square,color = red, mark size = 3pt ] %
        table[x=m,y=error,col sep=comma]{Data/RandomTestVaryingMsingle/testRandomSingle50.csv};
        \addlegendentry{$n = 50$}

        \addplot[only marks,mark = diamond,color = olive, mark size = 3pt ] %
        table[x=m,y=error,col sep=comma]{Data/RandomTestVaryingMsingle/testRandomSingle100.csv};
        \addlegendentry{$n = 100$}

        \addplot[only marks,mark = Mercedes star,color = purple, mark size = 3pt ] %
        table[x=m,y=error,col sep=comma]{Data/RandomTestVaryingMsingle/testRandomSingle200.csv};
        \addlegendentry{$n = 200$}

        \addplot[only marks,mark = o,color = violet, mark size = 3pt ] %
        table[x=m,y=error,col sep=comma]{Data/RandomTestVaryingMsingle/testRandomSingle400.csv};
        \addlegendentry{$n = 400$}

        \end{semilogyaxis}
        \end{tikzpicture}

\end{subfigure}
\centering
\begin{subfigure}
{0.496\textwidth}
 \begin{tikzpicture}[scale=0.67]
\begin{loglogaxis}[
    xtick={2,4,8,16},
xticklabels={$2^1$,$2^2$,$2^3$,$2^4$},
width=1.3\textwidth,
  height=1.2\textwidth,
grid=major,xlabel={number of parameters $m$},ylabel={ time in seconds},xmin=2, xmax=15,ymin = 1e-5,ymax = 1600,
legend style={at={(0.97,0.03)},anchor=south east}] 

    \addplot[only marks,mark = triangle,color = blue, mark size = 3pt ] %
    table[x=m,y=time,col sep=comma]{Data/RandomTestVaryingMsingle/testRandomSingleWC25.csv};
    \addlegendentry{$n = 25$}
    \addplot[sharp plot, color = blue] %
    table[x=m,y expr = 1.5*(25^2.5*x+x^3) *1e-7,col sep=comma, forget plot]{Data/RandomTestVaryingM/testRandom2.csv};

    \addplot[only marks,mark = square,color = red, mark size = 3pt ] %
    table[x=m,y=time,col sep=comma]{Data/RandomTestVaryingMsingle/testRandomSingleWC50.csv};
    \addlegendentry{$n = 50$}
    \addplot[sharp plot, color = red] %
    table[x=m,y expr = 1.5*(50^2.5*x+x^3) *1e-7,col sep=comma, forget plot]{Data/RandomTestVaryingM/testRandom2.csv};

    \addplot[only marks,mark = diamond,color = olive, mark size = 3pt ] %
    table[x=m,y=time,col sep=comma]{Data/RandomTestVaryingMsingle/testRandomSingleWC100.csv};
    \addlegendentry{$n = 100$}
    \addplot[sharp plot, color = olive] %
    table[x=m,y expr = 1.5*(100^2.5*x+x^3) *1e-7,col sep=comma, forget plot]{Data/RandomTestVaryingM/testRandom2.csv};

    \addplot[only marks,mark = Mercedes star,color = purple, mark size = 3pt ] %
    table[x=m,y=time,col sep=comma]{Data/RandomTestVaryingMsingle/testRandomSingleWC200.csv};
    \addlegendentry{$n = 200$}
    \addplot[sharp plot, color = purple] %
    table[x=m,y expr = 1.5*(200^2.5*x+x^3) *1e-7,col sep=comma, forget plot]{Data/RandomTestVaryingM/testRandom2.csv};

    \addplot[only marks,mark = o,color = violet, mark size = 3pt ] %
    table[x=m,y=time,col sep=comma]{Data/RandomTestVaryingMsingle/testRandomSingleWC400.csv};
    \addlegendentry{$n = 400$}
    \addplot[sharp plot, color = violet] %
    table[x=m,y expr = 1.5*(400^2.5*x+x^3) *1e-7,col sep=comma, forget plot]{Data/RandomTestVaryingM/testRandom2.csv};

\end{loglogaxis}
\end{tikzpicture}   

\end{subfigure}
\begin{subfigure}{0.496\textwidth}

    \begin{tikzpicture}[scale=0.67]
        
        \begin{semilogyaxis}[
        width=1.3\textwidth,
          height=1.2\textwidth,
        grid=major,xlabel={number of parameters $m$},ylabel={ error},xmin=2, xmax=15, 
        legend style={at={(0.97,0.03)},anchor=south east}]

        \addplot[only marks,mark = triangle,color = blue, mark size = 3pt ] %
        table[x=m,y=error,col sep=comma]{Data/RandomTestVaryingMsingle/testRandomSingleWC25.csv};
        \addlegendentry{$n = 25$}

        \addplot[only marks,mark = square,color = red, mark size = 3pt ] %
        table[x=m,y=error,col sep=comma]{Data/RandomTestVaryingMsingle/testRandomSingleWC50.csv};
        \addlegendentry{$n = 50$}

        \addplot[only marks,mark = diamond,color = olive, mark size = 3pt ] %
        table[x=m,y=error,col sep=comma]{Data/RandomTestVaryingMsingle/testRandomSingleWC100.csv};
        \addlegendentry{$n = 100$}

        \addplot[only marks,mark = Mercedes star,color = purple, mark size = 3pt ] %
        table[x=m,y=error,col sep=comma]{Data/RandomTestVaryingMsingle/testRandomSingleWC200.csv};
        \addlegendentry{$n = 200$}

        \addplot[only marks,mark = o,color = violet, mark size = 3pt ] %
        table[x=m,y=error,col sep=comma]{Data/RandomTestVaryingMsingle/testRandomSingleWC400.csv};
        \addlegendentry{$n = 400$}

        \end{semilogyaxis}
        \end{tikzpicture}

\end{subfigure}
\caption{Showcase of the performance of \Cref{alg:NewtonGlobalized} for $n\times n$ matrices and varying number of parameters $m$. The multiparameter eigenvalue problems were generated randomly as  described in \Cref{sec: randomly generated examples}. 
The first two graphs show the case of increasingly bad conditioned problems generated with Laguerre polynomials~\eqref{eq: Laguerre random matrices}, the second two graphs show the uniformly well conditioned problems~\eqref{eq: uniformly well conditioned matrices}
As a comparison, we show the observed rates for the costs~$mn^2+m^3$ and~$mn^{2.5}+m^3$, respectively.}\label{fig:expRandomVaryingMSingle}

\end{figure}
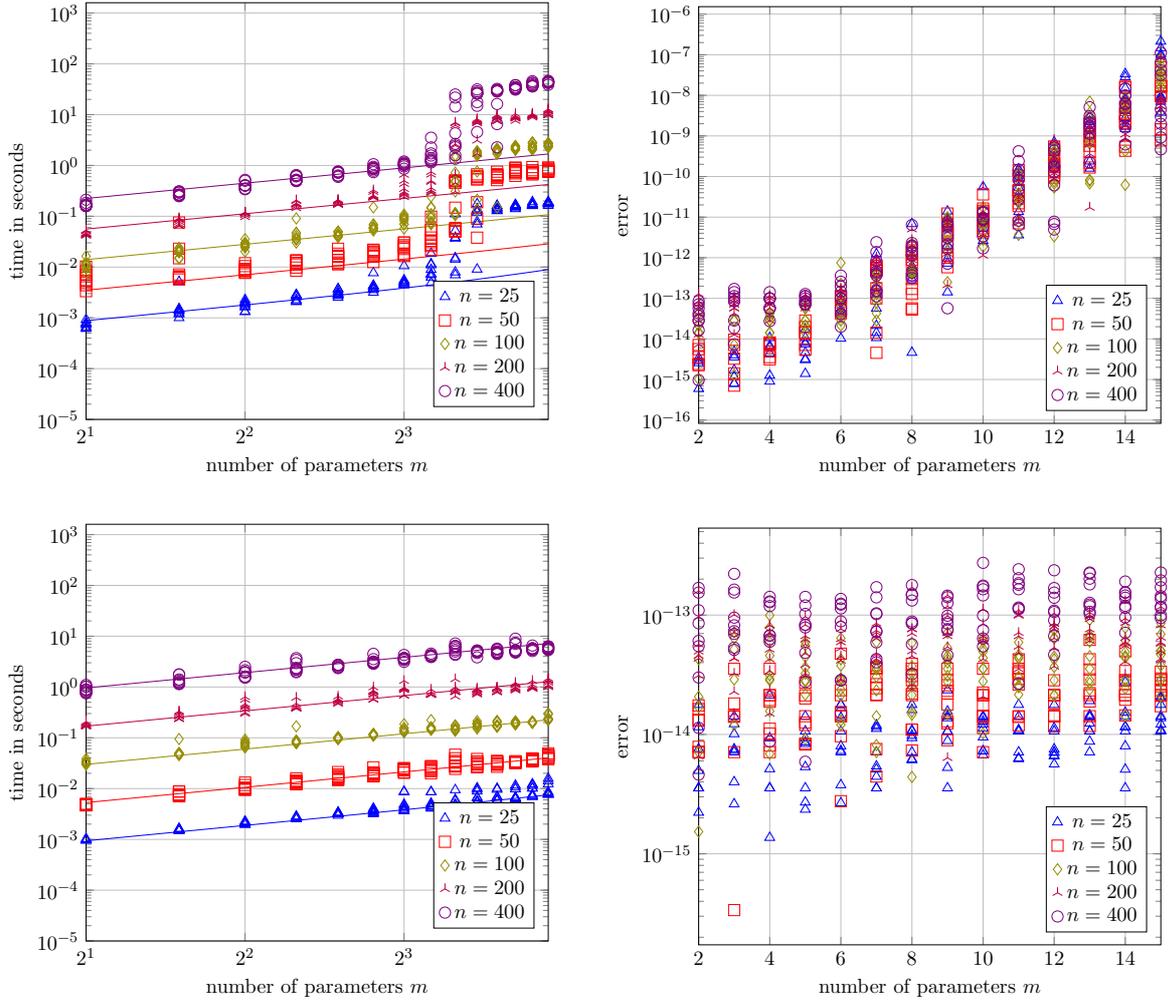

\subsection{Locally definite problem}\label{sec:localdefiniteexperiment}

At last, we apply our methods to a homogeneous  MEP of the form~\eqref{eq:HomMEP} that is locally definite but not definite. We are only aware of the following example from~\cite[\ch 1.5]{Volkmer88}. Let
{\allowdisplaybreaks
\begin{align*}
    A_{10}=A_{21}=A_{32}&=\begin{pmatrix}
      1 \\
      & 5\\
      && 1\\
      &&& 1
    \end{pmatrix},&
    A_{11}=A_{20}=A_{33}&=\begin{pmatrix}
      1 \\
      & 1\\
      && 5\\
      &&& 1
    \end{pmatrix},
    \\
    A_{12}=A_{23}=A_{30}&=\begin{pmatrix}
      5 \\
      & 1\\
      && 1\\
      &&& 1
    \end{pmatrix},&
    A_{13}=A_{22}=A_{31}&=-\begin{pmatrix}
      1 \\
      & 1\\
      && 1\\
      &&& 5
    \end{pmatrix}.
\end{align*}}
The resulting MEP is locally definite which can be seen by applying \Cref{lm:localdefinite}. Indeed, for every sign $\sigma \in\{-1,1\}^3$ one can choose $\alpha$ as one of the vectors $\pm e_i\in \R^4$. It is however not definite. To see this, note that 
\[
\begin{aligned}
\lambda_1&=\frac{1}{\sqrt{12}}(-1, -3, 1, 1),&\lambda_2&=\frac{1}{\sqrt{12}}(-1 ,1 ,1 ,-3 ),
\\
\lambda_3&=\frac{1}{\sqrt{12}}(-1, 1 ,-3, 1),& \lambda_4&=\frac{1}{\sqrt{12}}(3, 1 ,1 ,1)
\end{aligned}
\]
are eigenvalues in $\mathcal P^+$ and $\lambda_1+\lambda_2+\lambda_3+\lambda_4=0$. Hence, there is no $\mu$ such that $\mu^\T \lambda>0$ for all $\lambda\in\mathcal P^+$ and the MEP cannot be definite.
The eigenvectors are given by the coordinate vectors $e_{i_1}\otimes e_{i_2} \otimes e_{i_3} $. We used \Cref{alg:NewtonHom} to find all eigenvalues. For this problem, the eigenvalues were found up to machine precision after two iterations in the algorithm, as the problem essentially becomes solving a linear system, as all matrices are diagonal,
and therefore commute. After performing a common congruence transform, that is, replacing the matrices $A_{k\ell}$ by $B^{}_k A^{}_{k\ell} B^\He_k$ with nonorthogonal square matrices $B_k$, the matrices no longer commute but the eigenvalues of the MEP are unchanged. Again, \Cref{alg:NewtonHom} finds all eigenvalues but now requires more iterations.

We also perturbed the involved matrices slightly, and the algorithm was still able to find all eigenvalues. This indicates, that the simple adaptation \Cref{alg:NewtonHom} of \Cref{alg:Newton} is applicable to problems that are locally but not globally definite.

\section{Conclusion and outlook}

We have analyzed a Newton method for definite multiparameter eigenvalue problems, that computes eigenvalues of a given multiindex. We have shown local quadratic convergence
and global convergence for eigenvalues with extreme properties. In numerical experiments we demonstrate that the method can reliably compute any eigenvalue possibly when using a globalization strategy. 
The complexity per eigenvalue does not exceed the complexity of solving $m$ ordinary eigenvalue problems together with the costs of solving a linear system of $m$ equations.
We are therefore able to find some eigenvalues of multiparameter problems with comparably large $m$, where other methods based on the matrices~$\Delta_m$ in~\eqref{eq:deltaMatrix} fail due to complexity.

The search of eigenvalues by its multiindex can have more advantages. 
If the multiparameter eigenvalue problem comes from a multiparameter Sturm-Liouville
problem, it is possible to choose an adaptive discretization for different eigenvalues. 
This can possibly be combined with exterior error estimation, to further reduce the costs. 
This would be of particular interest in the case of localized eigenfunctions as it is for example observed in~\cite{ALSW_whispering_2014}. An adaptive discretization strategy would also lead to a very good initialization of forthcoming iterations.
Once in the proximity of an eigenvalue, one could also switch to matrix free methods. These can often, depending on the discretization, be applied with costs of order $\mathcal O(n)$ or $\mathcal O(n\log n)$. 

Another interesting topic of further research is concerned with the properties of multiparameter eigenvalue problems that are not quite definite but in some sense close. When a generalized eigenvalue problem is almost definite, that is, there is a linear combination of the involved matrices $\lambda A+\mu B$ that has only a few non-positive eigenvalues, one can show with the help of inertia laws that most of the eigenvalues are real~\cite{NN2019}.
Similarly, when a multiparameter problem is almost definite, many eigenvalues have a multiindex in the sense of \Cref{def:index}. These can again be targeted by a Newton's method in the same way. This will have additional challenges such as unknown existence of eigenvalues for multiindices and possibly non-uniqueness.

\subsection*{Acknowledgements}
I would like
to thank the anonymous referees for their helpful comments and suggestions, and Bor Plestenjak for providing me with a copy of~\cite{Bohte1982}.
The work of H.E.~was funded by Deutsche Forschungs\-gemeinschaft (DFG, German Research Foundation) – Projektnummer 501389786.

\bibliographystyle{amsplain}

\bibliography{main}

\providecommand{\bysame}{\leavevmode\hbox to3em{\hrulefill}\thinspace}
\providecommand{\MR}{\relax\ifhmode\unskip\space\fi MR }
\providecommand{\MRhref}[2]{%
  \href{http://www.ams.org/mathscinet-getitem?mr=#1}{#2}
}
\providecommand{\href}[2]{#2}
\begin{thebibliography}{10}

\bibitem{ALSW_whispering_2014}
P.~Amodio, T.~Levitina, G.~Settanni, and E.~B. Weinm\"uller, \emph{Numerical simulation of the whispering gallery modes in prolate spheroids}, Comput. Phys. Commun. \textbf{185} (2014), no.~4, 1200--1206.

\bibitem{Arscott83}
F.~M. Arscott, P.~J. Taylor, and R.~V.~M. Zahar, \emph{On the numerical construction of ellipsoidal wave functions}, Math. Comp. \textbf{40} (1983), no.~161, 367--380.

\bibitem{Atkinson1972}
F.~V. Atkinson, \emph{Multiparameter eigenvalue problems}, Mathematics in Science and Engineering, Vol. 82, Academic Press, New York-London, 1972, Volume I: Matrices and compact operators.

\bibitem{AtkinsonMingarelli2011}
F.~V. Atkinson and A.~B. Mingarelli, \emph{Multiparameter eigenvalue problems}, CRC Press, Boca Raton, FL, 2011, Sturm-Liouville theory.

\bibitem{BindingBrwon1984}
P.~Binding and P.~J. Browne, \emph{Multiparameter {S}turm theory}, Proc. Roy. Soc. Edinburgh Sect. A \textbf{99} (1984), no.~1-2, 173--184.

\bibitem{Bohte1982}
Z.~Bohte, \emph{Numerical solution of some two-parameter eigenvalue problems}, Anton Kuhelj Memorial Volume, Slov. Acad. Sci. Art., Ljubljana (1982).

\bibitem{Clarke_Book_Opt}
F.~H. Clarke, \emph{Optimization and nonsmooth analysis}, second ed., Society for Industrial and Applied Mathematics (SIAM), Philadelphia, PA, 1990.

\bibitem{Coddington55}
E.~A. Coddington and N.~Levinson, \emph{Theory of ordinary differential equations}, McGraw-Hill, New York, 1955.

\bibitem{DYY2016}
B.~Dong, B.~Yu, and Y.~Yu, \emph{A homotopy method for finding all solutions of a multiparameter eigenvalue problem}, SIAM J. Matrix Anal. Appl. \textbf{37} (2016), no.~2, 550--571.

\bibitem{eisenmann}
H.~Eisenmann, \emph{Multilinear optimization in low-rank models}, Ph.D. thesis, Universit{\"a}t Leipzig, 2023.

\bibitem{eisenmann2021solving}
H.~Eisenmann and Y.~Nakatsukasa, \emph{Solving two-parameter eigenvalue problems using an alternating method}, Linear Algebra Appl. \textbf{643} (2022), 137--160.

\bibitem{Faierman1991}
M.~Faierman, \emph{Two-parameter eigenvalue problems in ordinary differential equations}, Pitman Research Notes in Mathematics Series, vol. 205, Longman Scientific \& Technical, Harlow; copublished in the United States with John Wiley \& Sons, Inc., New York, 1991.

\bibitem{GHP_Math_2012}
C.~I. Gheorghiu, M.~E. Hochstenbach, B.~Plestenjak, and J.~Rommes, \emph{Spectral collocation solutions to multiparameter {M}athieu's system}, Appl. Math. Comput. \textbf{218} (2012), no.~24, 11990--12000.

\bibitem{HKP2024}
H.~He, D.~Kressner, and B.~Plestenjak, \emph{Randomized methods for computing joint eigenvalues, with applications to multiparameter eigenvalue problems and root finding.}, Numer. Algor. (2024).

\bibitem{HKP2004}
M.~E. Hochstenbach, T.~Ko\v{s}ir, and B.~Plestenjak, \emph{A {J}acobi-{D}avidson type method for the two-parameter eigenvalue problem}, SIAM J. Matrix Anal. Appl. \textbf{26} (2004/05), no.~2, 477--497.

\bibitem{HMMP2019}
M.~E. Hochstenbach, K.~Meerbergen, E.~Mengi, and B.~Plestenjak, \emph{Subspace methods for three-parameter eigenvalue problems}, Numer. Linear Algebra Appl. \textbf{26} (2019), no.~4, e2240, 22.

\bibitem{HP2002}
M.~E. Hochstenbach and B.~Plestenjak, \emph{A {J}acobi-{D}avidson type method for a right definite two-parameter eigenvalue problem}, SIAM J. Matrix Anal. Appl. \textbf{24} (2002), no.~2, 392--410.

\bibitem{ISNewtonbook2014}
A.~F. Izmailov and M.~V. Solodov, \emph{Newton-type methods for optimization and variational problems}, Springer Series in Operations Research and Financial Engineering, Springer, Cham, 2014.

\bibitem{JH2009}
E.~Jarlebring and M.~E. Hochstenbach, \emph{Polynomial two-parameter eigenvalue problems and matrix pencil methods for stability of delay-differential equations}, Linear Algebra Appl. \textbf{431} (2009), no.~3-4, 369--380.

\bibitem{Ji_1992_twoparabisection}
X.~Ji, \emph{A two-dimensional bisection method for solving two-parameter eigenvalue problems}, SIAM J. Matrix Anal. Appl. \textbf{13} (1992), no.~4, 1085--1093.

\bibitem{Kato1976}
T.~Kato, \emph{Perturbation theory for linear operators}, second ed., Grundlehren der Mathematischen Wissenschaften, Band 132, Springer-Verlag, Berlin-New York, 1976.

\bibitem{Levitina1994}
T.~Levitina, \emph{On numerical solution of multiparameter {S}turm-{L}iouville spectral problems}, Numerical analysis and mathematical modelling, Banach Center Publ., vol.~29, Polish Acad. Sci. Inst. Math., Warsaw, 1994, pp.~275--281.

\bibitem{Levititna1999}
T.~V. Levitina, \emph{On a numerical solution of some three-parameter spectral problems}, Zh. Vychisl. Mat. Mat. Fiz. \textbf{39} (1999), no.~11, 1787--1801.

\bibitem{MP2015}
K.~Meerbergen and B.~Plestenjak, \emph{A {S}ylvester-{A}rnoldi type method for the generalized eigenvalue problem with two-by-two operator determinants}, Numer. Linear Algebra Appl. \textbf{22} (2015), no.~6, 1131--1146.

\bibitem{Meixner1954}
J.~Meixner and F.~W. Sch\"{a}fke, \emph{Mathieusche {F}unktionen und {S}ph\"{a}roidfunktionen mit {A}nwendungen auf physikalische und technische {P}robleme}, Die Grundlehren der mathematischen Wissenschaften in Einzeldarstellungen mit besonderer Ber\"{u}cksichtigung der Anwendungsgebiete, Band LXXI, Springer-Verlag, Berlin-G\"{o}ttingen-Heidelberg, 1954.

\bibitem{NN2019}
Y.~Nakatsukasa and V.~Noferini, \emph{Inertia laws and localization of real eigenvalues for generalized indefinite eigenvalue problems}, Linear Algebra Appl. \textbf{578} (2019), 272--296.

\bibitem{multipareig}
B.~Plestenjak, \emph{Multi{P}ar{E}ig}, \url{https://www.mathworks.com/matlabcentral/fileexchange/47844-multipareig}, MATLAB Central File Exchange. Retrieved December 18, 2023.

\bibitem{multipareignew}
\bysame, \emph{Multi{P}ar{E}ig}, \url{https://www.mathworks.com/matlabcentral/fileexchange/47844-multipareig}, MATLAB Central File Exchange. Retrieved January 1, 2025.

\bibitem{Plestenjak2000}
B.~Plestenjak, \emph{A continuation method for a right definite two-parameter eigenvalue problem}, SIAM J. Matrix Anal. Appl. \textbf{21} (2000), no.~4, 1163--1184.

\bibitem{Plestenjak2001}
\bysame, \emph{A continuation method for a weakly elliptic two-parameter eigenvalue problem}, IMA J. Numer. Anal. \textbf{21} (2001), no.~1, 199--216.

\bibitem{Plestenjak2015}
B.~Plestenjak, C.~I. Gheorghiu, and M.~E. Hochstenbach, \emph{Spectral collocation for multiparameter eigenvalue problems arising from separable boundary value problems}, J. Comput. Phys. \textbf{298} (2015), 585--601.

\bibitem{RJ2021}
E.~Ringh and E.~Jarlebring, \emph{Nonlinearizing two-parameter eigenvalue problems}, SIAM J. Matrix Anal. Appl. \textbf{42} (2021), no.~2, 775--799.

\bibitem{RodriguezDuYouLimFiber2021}
J.~I. Rodriguez, J.-H. Du, Y.~You, and L.-H. Lim, \emph{Fiber product homotopy method for multiparameter eigenvalue problems}, Numer. Math. \textbf{148} (2021), no.~4, 853--888.

\bibitem{SNTI2016}
S.~Sakaue, Y.~Nakatsukasa, A.~Takeda, and S.~Iwata, \emph{Solving generalized {CDT} problems via two-parameter eigenvalues}, SIAM J. Optim. \textbf{26} (2016), no.~3, 1669--1694.

\bibitem{SW1979}
D.~Schmidt and G.~Wolf, \emph{A method of generating integral relations by the simultaneous separability of generalized schrödinger equations}, SIAM Journal on Mathematical Analysis \textbf{10} (1979), no.~4, 823--838.

\bibitem{Sleeman1978}
B.~D. Sleeman, \emph{Multiparameter spectral theory in {H}ilbert space}, Research Notes in Mathematics, vol.~22, Pitman (Advanced Publishing Program), Boston, Mass.-London, 1978.

\bibitem{ST1986}
T.~Slivnik and G.~Tom\v{s}i\v{c}, \emph{A numerical method for the solution of two-parameter eigenvalue problems}, J. Comput. Appl. Math. \textbf{15} (1986), no.~1, 109--115.

\bibitem{SunSun2002}
D.~Sun and J.~Sun, \emph{Strong semismoothness of eigenvalues of symmetric matrices and its application to inverse eigenvalue problems}, SIAM J. Numer. Anal. \textbf{40} (2002), no.~6, 2352--2367 (2003).

\bibitem{Teschl_oscillation_1996}
G.~Teschl, \emph{Oscillation theory and renormalized oscillation theory for {J}acobi operators}, J. Differential Equations \textbf{129} (1996), no.~2, 532--558. \MR{1404392}

\bibitem{Teschl2012}
\bysame, \emph{Ordinary differential equations and dynamical systems}, Graduate Studies in Mathematics, vol. 140, American Mathematical Society, Providence, RI, 2012.

\bibitem{Tuy2016}
H.~Tuy, \emph{Convex analysis and global optimization}, Springer Optimization and Its Applications, vol. 110, Springer, [Cham], 2016.

\bibitem{Volkmer88}
H.~Volkmer, \emph{Multiparameter eigenvalue problems and expansion theorems}, Lecture Notes in Mathematics, vol. 1356, Springer-Verlag, Berlin, 1988.

\end{thebibliography}

\end{document}